\documentclass[11pt,reqno,final]{amsart}%
\usepackage[a4paper,left=2.8cm,right=2.8cm,marginparwidth=2cm,top=2.8cm,bottom=3.5cm]{geometry}%
\usepackage[textsize=tiny]{todonotes}%
\usepackage[hidelinks,unicode=true]{hyperref}%
\usepackage{mathtools, graphicx, lmodern, enumitem, amssymb, bbm, xcolor, comment}%
\usepackage{subcaption}
\usepackage[american]{babel}%
\graphicspath{{figs/}}%
\allowdisplaybreaks%
\numberwithin{equation}{section}%
\title[When is a radial Riesz equilibrium supported on a sphere?]{Riesz energy with a radial external field:\\when is the equilibrium support a sphere?}
\author[Chafaï]{D.\ Chafaï}%
\address[Djalil Chafaï]{DMA, École normale supérieure -- PSL, 45 rue d'Ulm, %
  F-75230 Cedex 5 Paris, France}%
\email{\url{djalil@chafai.net}}%
\urladdr{\url{http://djalil.chafai.net/}}%
\author[Matzke]{R.\ W.\ Matzke}%
\address[Ryan W.\ Matzke]{Center for Constructive Approximation, %
  1326 Stevenson Center, Vanderbilt University, Nashville, TN 37240, USA}%
\email{\url{ryan.w.matzke@vanderbilt.edu}}%
\urladdr{\url{http://www.ryanmatzke.com/}}%
\author[Saff]{E.\ B.\ Saff}%
\address[Edward B.\ Saff]{Center for Constructive Approximation, %
  1326 Stevenson Center, Vanderbilt University, Nashville, TN 37240, USA}%
\email{\url{edward.b.saff@vanderbilt.edu}}%
\urladdr{\url{https://my.vanderbilt.edu/edsaff/}}%
\author[Vu]{M.\ Q.\ H.\ Vu}%
\address[Minh Quan H.\ Vu]{Center for Constructive Approximation, %
  1326 Stevenson Center, Vanderbilt University, Nashville, TN 37240, USA}%
\email{\url{minh.quan.h.vu@vanderbilt.edu}}%
\author[Womersley]{R.\ S.\ Womersley}%
\address[Robert S.\ Womersley]{School of Mathematics and Statistics, %
  UNSW Sydney, %
  Sydney NSW 2052, Australia}%
\email{\url{R.Womersley@unsw.edu.au}}%
\urladdr{\url{https://web.maths.unsw.edu.au/~rsw/}}%
\newtheorem{theorem}{Theorem}[section]%
\newtheorem{lemma}[theorem]{Lemma}%
\newtheorem{corollary}[theorem]{Corollary}%
\newtheorem{proposition}[theorem]{Proposition}%
\newtheorem{remark}[theorem]{Remark}%
\definecolor{darkgreen}{rgb}{0.0, 0.5, 0.0}

\newcommand{\RW}[1]{{\color{magenta}#1}}

\DeclareMathOperator{\supp}{supp}

\DeclareMathOperator{\Minz}{minimize}

\newcommand{\dR}{\mathbb{R}}
\newcommand{\dS}{\mathbb{S}}

\newcommand{\cP}{\mathcal{P}}

\newcommand{\rF}{\mathrm{F}}

\newcommand{\rd}{\mathrm{d}}
\newcommand{\mueq}{\mu_{\mathrm{eq}}}
\newcommand{\HG}[2]{{}_{#1}\rF_{#2}} 
\newcommand{\hsd}[1]{h_{s,d} \left( #1 \right)}
\newcommand{\hsdp}[1]{h_{s,d}' \left( #1 \right)}
\newcommand{\hsdpp}[1]{h_{s,d}'' \left( #1 \right)}
\newcommand{\hzd}[1]{h_{0,d} \left( #1 \right)}
\newif\ifAddSuffCond
\AddSuffCondtrue
\begin{document}

\begin{abstract}
We consider Riesz energy problems with radial external fields. 
We study the question of whether or not the equilibrium is the uniform distribution on a sphere.
We develop general necessary as well as general sufficient conditions on the external field that apply to
powers of the Euclidean norm as well as certain Lennard--Jones type fields. Additionally, in the former case, we completely characterize the values  of the power for which dimension reduction occurs in the sense that the support of the equilibrium measure becomes a sphere. 
We also briefly discuss the relation between these problems and certain
constrained optimization problems.
Our approach involves the Frostman characterization, the Funk--Hecke formula, and the calculus of hypergeometric functions.
\end{abstract}

\maketitle


\section{Introduction}

\subsection{Background}

We consider Riesz equilibrium problems with an external field on the whole Euclidean space $\dR^d$ and, unless otherwise stated, assume $d \geq 2$.
The Riesz $s$-kernel\RW{\footnote{Our definition of $K_s$, for $s \neq 0$, differs from the more traditional version in that it includes $s$ in its denominator. This has the advantage to produce formulas which are continuous as $s=0$.}}
$K_s:\dR^d 
\to (- \infty, \infty]$ is defined by
\begin{equation*}
  K_s(x) := \begin{cases}
    \displaystyle\frac{1}{s \|x\|^s} & \text{if $s \neq 0$} \\
    \displaystyle-\log(\|x\|) & \text{if $s = 0$}
  \end{cases},  
\end{equation*}
where $\|x\|:=\sqrt{x_1^2+\cdots+x_d^2}=\sqrt{\langle x,x\rangle}$ is the Euclidean norm. 
This kernel is continuous for $s < 0$,  singular but integrable when $0\leq s < d$, and hypersingular (non-integrable) for $s \geq d$.
Throughout, we assume $-2 < s < d$,
which ensures that the kernel is
integrable and conditionally strictly positive definite on compact sets,
see \cite{BorHS19}.

We assume throughout that our external field $V:\dR^d\to(-\infty,+\infty]$ is radial, of the form
\begin{equation} \label{eq:v}
  V(x) = v(r^2), \quad r := \|x\|,
\end{equation}
where $v:[0,+\infty)\to (-\infty,+\infty]$ is lower semi-continuous, bounded from below, and finite on some interval $(a,b)$.  

Let $\cP(\dR^d)$ be the set of probability measures on $\dR^d$.
For $s < d$, the \emph{energy} of $\mu\in \cP(\dR^d)$ is defined by
\begin{equation}\label{eq:Energy Def}
  I_{s,V}(\mu)
  :=\iint\big( K_s(x-y) + V(x) + V(y)\big)\rd\mu(x)\rd\mu(y)\in(-\infty,+\infty].
\end{equation}
When $V \equiv 0$, we simply write $I_{s}(\mu)$. We shall also denote by $\mathcal{W}_{s,V}$ the minimum of the energy over all probability measures, which is known as the \emph{Wiener constant}, 
\begin{equation*}
\mathcal{W}_{s,V} := \inf_{\mu \in \cP(\dR^d)} I_{s,V}(\mu).
\end{equation*}
When they exist, the minimizers, called \emph{equilibrium measures}, are denoted by $\mueq$. When unique, these measures must have radially symmetric support, due to the radial symmetry of $K_s$ and $V$.

The \emph{potential} of $\mu$ is the locally integrable function $U_s^\mu: \dR^d \to (-\infty, +\infty]$ defined by
\begin{equation*}
  U_s^{\mu}(x) := \int K_s(x-y)\rd\mu(y).
\end{equation*}
We will call the quantity $U_s^\mu+V$ the \emph{modified potential} of $\mu$.

We concentrate on
confining potentials where the external field ensures that the support of the equilibrium measure is compact.
The following is a list of some sufficient conditions that
ensure the equilibrium measure exists and has compact support:

\begin{enumerate}[label=(\alph*)]
\item
$0 < s < d$ and
  either
   $\displaystyle v(\infty) := \lim_{r\to \infty} v(r^2) = +\infty$,\\
 or $v(\infty)
 < +\infty$ and
$\displaystyle
    \lim_{r \rightarrow \infty} sr^s \big( v(r^2) - v(\infty) \big)< -1, 
$

  \item 
  $s = 0$ and
  $\displaystyle\lim_{r\to\infty}\left( v(r^2)-\log r \right)=+\infty$,
\item
$-2 < s < 0$ and 
  $\displaystyle\limsup_{r\to\infty} sr^s v(r^2)<-2^{-s}$.
\end{enumerate}

\noindent
Moreover, in each of these cases $\mathcal{W}_{s,V} < \infty$ and the equilibrium measure $\mueq$ is unique.
See, for example, \cite[Theorems 2.1 and 2.4]{DraOSW23}, \cite[Theorem I.1.3]{SafT97}, and \cite[Corollary 4.4.16]{BorHS19}.
Furthermore, $\mu = \mueq$ is characterized by the
variational \emph{Frostman conditions}:
for some finite constant $C$,
        \begin{align}\label{cond: Frostman}
            \begin{split}
                U_s^{\mu} + V \geq C, & \quad\text{q.e. on } \mathbb{R}^d,\\
                U_s^{\mu} + V \leq C, & \quad\text{on } \supp \mu ,
            \end{split}
        \end{align}
        where q.e.\ denotes a property holding except on sets of $s$-capacity zero
        (see \cite{Lan72,BorHS19}).

\subsection{Main results}

In the statement of our results we use the following notation:  $\rd x$ denotes Lebesgue measure in $\dR^d$ and $\sigma_R$ is the uniform
probability measure on the sphere $\dS^{d-1}_R: = \{x\in\dR^d:\|x\|=R\}$ in $\mathbb{R}^d$.
For $R = 1$, we further put $\sigma := \sigma_1$ and $\dS^{d-1}:=\dS^{d-1}_1$. 
For an interval $I \subseteq \mathbb{R}$, a function $f: I \rightarrow \mathbb{R} \cup \{ \pm \infty\}$ is $\mathcal{C}^k$ \emph{in the extended sense on $I$} when for each $\ell \in \{0, ..., k\}$ and $y \in I$,  $f^{(\ell)}(y)$ exists as an element of $\mathbb{R} \cup \{ \pm\infty\}$ and is finite except at a finite set of values in $I$, and $\lim\limits_{x \rightarrow y} f^{(\ell)}(x) = f^{(\ell)}(y)$, with the limit being one-sided for the endpoints of $I$. When we do not specify the interval, it is assumed to be $[0, \infty)$.


Our first result below gives, as a special case, conditions when the support of equilibrium measure $\mueq$ 
for the energy \eqref{eq:Energy Def} cannot be a sphere.

\begin{theorem}[Structure of $\mueq$ for $s\geq d-3$]
\label{thm:Sphere not a minimizer}
 Suppose that $ d-3 \leq s < d$ and $v$ is bounded from below, $\mathcal{C}^2$ in the extended sense, with $v''$ finite on $(0, \infty)$, and such that $\mathrm{I}_{s,V}$ has a unique equilibrium measure $\mueq$. If $d = 2$ and $-1 \leq s < 0$, assume also that $\lim\limits_{\rho \rightarrow 0^+} \rho^{\frac{s}{2}+1} v'(\rho) =0$. 
With $x = r \theta$, $r \in [0, \infty)$, and $\theta \in \mathbb{S}^{d-1}$, let $\nu \in \mathcal{P}\big([0,\infty)\big)$ be such that
\begin{equation*}
\rd\mueq(x) = \rd\sigma(\theta)\rd\nu(r).
\end{equation*}
 Then $\supp(\nu)$ is a perfect set, i.e.\ closed with no isolated points.
\end{theorem}

In particular $\mathcal{H}_{d-1}(\supp(\mueq)) = \infty$, where $\mathcal{H}_{d-1}$ is the $(d-1)$-dimensional Hausdorff measure. The additional assumption for $d =2$, $-1 \leq s < 0$ is necessary to rule out the possibility of an isolated point at $0$, as when $ s < 0$, it is possible for $\delta_0$ to be the equilibrium measure, as shown in \cite[Remark 1.1]{ChaSW23}. We also note that the conditions on $v$ can be weakened if we assume $s \geq d-2$, as we show in Lemma \ref{lem:perfect set as a minimizer}.

When $-2<s<d-3$, a dimension reduction phenomenon may occur. In particular, the following result addresses this possibility when the external field is a power of the Euclidean norm. For this case, our results completely answer the question of when $\mueq$ is supported on a sphere, as encountered in \cite[Theorem 17]{HerGBS24} and \cite[Theorem 1.2]{ChaSW23}. For the statement of this result, we introduce the following constants, involving the classical hypergeometric functions $\HG21$ and $\HG32$ (see Appendix~\ref{se:appendix HG}) : 
\begin{equation}\label{eq:c_s,d definition}
    c_{s,d} := 
        \HG21\Bigr(\frac{s}{2}, \frac{2+s-d}{2}; \frac{d}{2}; 1\Bigr)%
        =\frac{\Gamma(\frac{d}{2})\Gamma(d -s-1)}
        {\Gamma(\frac{d-s}{2})\Gamma(d-\frac{s}{2}-1)},
        \quad\text{for $-2 < s < d-1$}.
\end{equation}
Note that $c_{s,d} > 0$, $c_{s,d}$ is a strictly convex function of $s$ on $-2 < s < d-1$,
and $c_{0,d} = 1$, $c_{d-2,d} = 1$ so $c_{s,d} < 1$ for $0 < s< d-2$.
For $s=0$, we also define, with $\psi_0$ denoting the digamma function, 
\begin{equation}\label{eq:b_d definition}
    b_{d} := 
       - \log(2) + \frac{1}{4}\; \HG32\Bigr(1,1, \frac{d+1}{2}; 2 , d; 1\Bigr) %
       = - \log(2) + \frac{1}{2}\psi_0(d-1) - \frac{1}{2}\psi_0\Bigr(\frac{d-1}{2}\Bigr).
\end{equation}

\begin{theorem}[Power-law external fields for which $\supp(\mueq)$ is a sphere]\label{thm:Sphere Min}
Suppose that $-2 < s < d-3$ and $V(x) = \frac{\gamma}{\alpha}\| x \|^{\alpha}$, where $\gamma>0$ and $\alpha > \max\{-s, 0\}$. Define
\begin{equation}\label{eq:Alpha Bound for Sphere}
\alpha_{s,d}:= \begin{cases} \max\left\{ \dfrac{sc_{s,d}}{2-2c_{s,d}}, \  2- \dfrac{(s+2)(d-s-4)}{2(d-s-3)}  \right\} & s \neq 0 \\[1.1em]
\max\left\{-\dfrac{1}{2b_d}, \   2- \dfrac{(d-4)}{(d-3)}  \right\} & s = 0
\end{cases}.
\end{equation}
If $\alpha \geq \alpha_{s,d}$,  then 
$\mueq = \sigma_{R_*}$, where
\begin{equation}\label{eq:radius for PL}
R_* = \left( \frac{c_{s,d}}{2 \gamma}  \right)^{\frac{1}{\alpha + s}}
= \left( \frac{\Gamma(\frac{d}{2}) \Gamma(d-s-1)}{ 2 \gamma \Gamma( \frac{d-s}{2}) \Gamma(d - \frac{s}{2}-1)}
\right)^{\frac{1}{\alpha + s}} .
\end{equation}
The energy is then
\begin{equation}\label{eq:energy for PL}
I_{s, V}(\sigma_{R_*}) =  \begin{cases}
 \dfrac{ (\alpha + s) (2\gamma)^{\frac{s}{\alpha+s}}}{\alpha s}  \left( \dfrac{\Gamma(\frac{d}{2}) \Gamma(d-s-1)}{ \Gamma( \frac{d-s}{2}) \Gamma(d - \frac{s}{2}-1)} \right)^{\frac{\alpha}{\alpha+s}} & s \neq 0 \\[2ex]
\dfrac{1 + \log(2\gamma)}{\alpha} - \log(2) + \frac{1}{2} \big( \psi_0 (d-1) - \psi_0( \frac{d-1}{2}) \big) & s = 0
\end{cases}.
\end{equation}
Furthermore, the threshold $\alpha_{s,d}$ is a sharp bound for $\alpha$, meaning that if $\max\{ -s,0 \}<\alpha< \alpha_{s,d}$, then for all $R>0$, $\sigma_{R}$ is not a minimizer of $I_{s,V}$.
\end{theorem}

The hypotheses require that $\alpha + s > 0$. When $s < 0$, $\alpha + s = 0$, and $\gamma > 1$, the equilibrium measure is a point mass at the origin,
otherwise the equilibrium measure does not exist, see \cite[Remark 1.1]{ChaSW23}.
The lower bounds \eqref{eq:Alpha Bound for Sphere} on the parameter $\alpha$ are illustrated in the four graphics of Figure~\ref{fig:Ralpha}.
The active bound in \eqref{eq:Alpha Bound for Sphere} changes at $s = d - 4$, for which $\alpha_{s,d} = 2$, and as $d$ increases the bound on $\alpha$ goes to $0$ for $-2 < s < d - 4$.
We remark that the properties of $c_{s,d}$ ensure that $\alpha_{s,d} > \max\{-s,0\}$ and
$\alpha_{s,d}$ is continuous at $s=0$, 
see the discussion at the beginning of section \ref{subsec:Proof for power law external field}.

\begin{figure}[htbp]
  \centering
  \includegraphics[width=\textwidth]{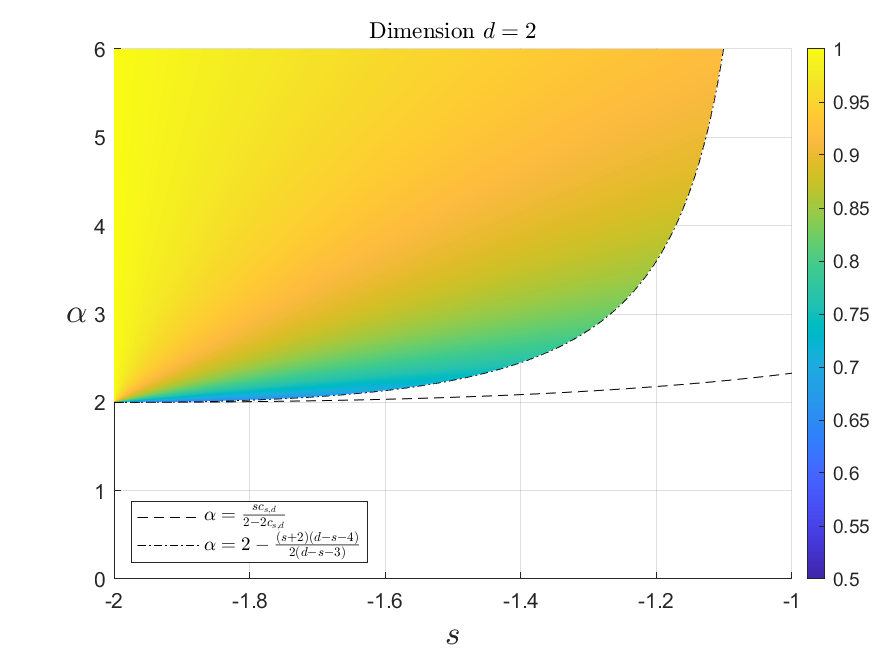}
  \includegraphics[width=\textwidth]{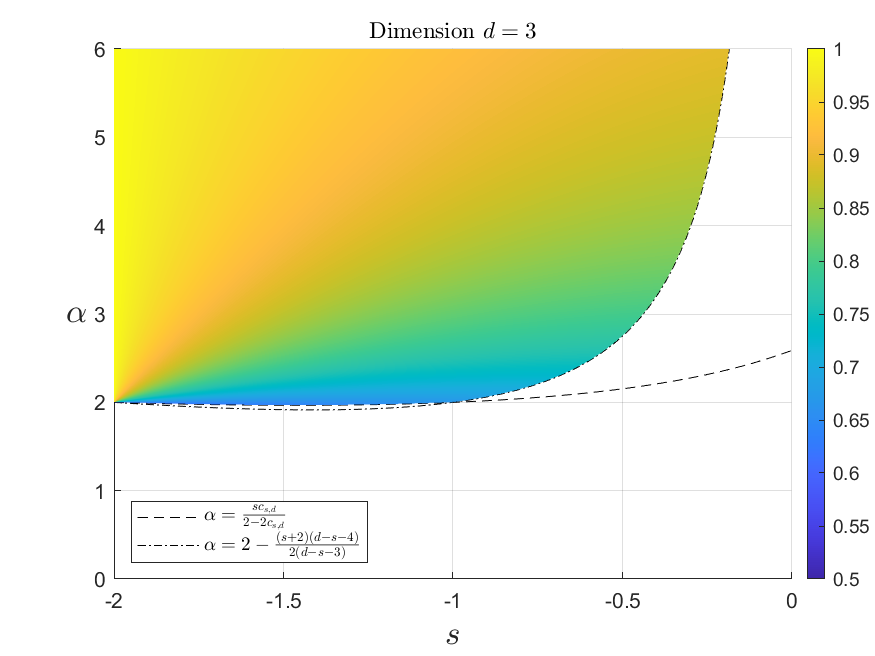}
  \caption{\label{fig:Ralpha} Graphics for four values of $d$ (see also next page). The color gives the value of $R_*$
  when the equilibrium support is $\mathbb{S}_{R_*}^{d-1}$,
  for $d = 2, 3, 4, 10$.
  The plots are for Riesz parameter $-2 < s < d-3$ and the external field
  power $\alpha\geq\alpha_{s,d}$,
  as in Theorem~\ref{thm:Sphere Min} with $\gamma = 1$.
  The two curves are the terms in the
  maximum in \eqref{eq:Alpha Bound for Sphere}.
  Outside the colored region the support is not a sphere.}
\end{figure}

\begin{figure}[htbp]\ContinuedFloat
  \centering
  \includegraphics[width=\textwidth]{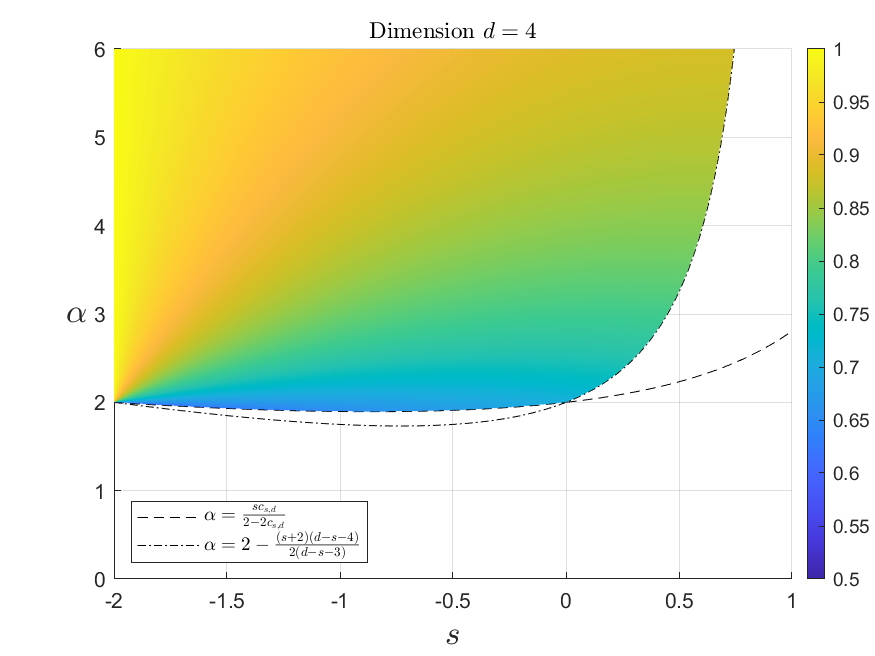}
  \includegraphics[width=\textwidth]{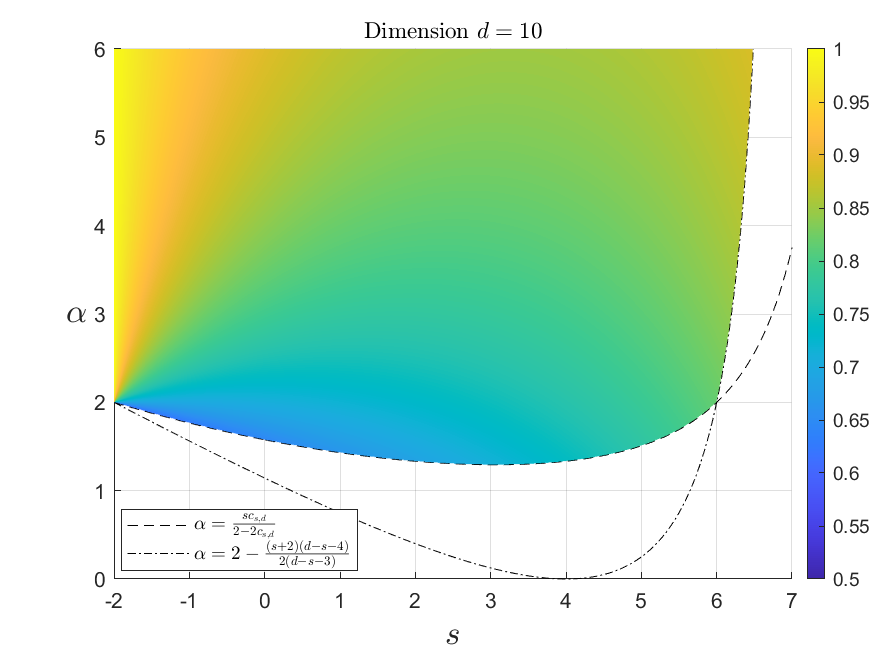}
\end{figure}

We next provide necessary conditions for a general twice continuously differentiable radial external field to yield $\sigma_{R_*}$ as the equilibrium measure.
\begin{theorem}[Necessary conditions]\label{thm:Necessary Condition Riesz}
      Suppose that $-2 < s < d-3$ and $V(x)=v(\|x\|^2)$, where $v(\cdot)$ is $\mathcal{C}^2$ in the extended sense. If $\mueq=\sigma_{R_*}$ for some $R_* > 0$, then the following must hold for $R = R_*$:
      \begin{enumerate}[label=\textnormal{(\roman*)}] 
        \item\label{item:NecCond SP} $\displaystyle R^{s+2}v'(R^2) = \frac{c_{s,d}}{4}$\vspace{.5em} 
        \item\label{item:NecCond DD} $\displaystyle R^2\frac{v''(R^2)}{v'(R^2)}\geq -\frac{(s+2)(d-s-4)}{4(d-s-3)}$\vspace{.5em} 
        \item\label{item:NecCond BC0} 
        if $s \neq 0$, then
        \[
           \lim_{r\to 0^+} R^s (v(r^2) - v(R^2)) \geq \frac{c_{s,d}-1}{s}
        \]
        if $s = 0$, then
        \[
           \lim_{r\to 0^+} v(r^2) - v(R^2) \geq b_d.
        \]
        \item\label{item:NecCond BCinf} if $s \neq 0$, then
        \[
          v(R^2) + \frac{R^{-s}}{s} \; c_{s,d} \leq
          \lim_{r\to\infty} \left[\frac{(R+r)^{-s}}{s} + v(r^2) \right]
        \]
        if $s = 0$, then
        \[
         - \log(R) + b_d + v(R^2)  \leq
         \lim_{r\to\infty} \left[- \log(R + r) + v(r^2) \right].
        \]
      \end{enumerate}
\end{theorem}
In Theorem~\ref{thm:Necessary Condition Riesz},
condition \ref{item:NecCond SP} arises from the requirement
that $R_*$ be a stationary point of the modified potential, while
\ref{item:NecCond DD} corresponds to non-negativity of a second derivative at $R_*$.
Conditions \ref{item:NecCond BC0} and \ref{item:NecCond BCinf} arise from boundary conditions at the origin and infinity, respectively.
Note that \ref{item:NecCond SP} may have no solutions or more than one solution,
as shown in Appendix~\ref{Sec:LJEF num ex}.
Additionally,
\ref{item:NecCond BC0} is trivially satisfied
when $\lim_{r\to 0^+} v(r^2) = \infty$, while
\ref{item:NecCond BCinf} is trivially satisfied
when $s > 0$ and $\lim_{r\to\infty} v(r^2) = \infty$.

\begin{theorem}[Sufficient conditions]\label{thm:Sufficient Conditions Riesz}
    Suppose that $-2 < s \leq d-4$, and $v(\cdot)$ is $\mathcal{C}^2$ in the extended sense, and such that $v''(r) \geq 0$ for all $r \in [0, \infty)$.
    If there exists $R_* \in (0, \infty)$ that satisfies
    Theorem~\ref{thm:Necessary Condition Riesz}\ref{item:NecCond SP},
    then $\mueq=\sigma_{R_*}$.
\end{theorem}

As some examples, Theorem~\ref{thm:Sufficient Conditions Riesz} ensures that $\mueq=\sigma_{R_*}$ for the following external fields:

\begin{itemize} 
\item \textbf{Lennard\,--\,Jones type:} 
$V(x) =  \frac{\gamma}{\alpha} \| x \|^{\alpha} - \frac{\gamma \eta}{\beta} \| x \|^{\beta}$,
where $\gamma, \eta > 0$, so that
\begin{equation*} 
        v(\rho) = \frac{\gamma}{\alpha} \rho^{\frac{\alpha}{2}} -
                  \frac{\gamma \eta}{\beta} \rho^{\frac{\beta}{2}}.
\end{equation*}
Theorem~\ref{thm:Sufficient Conditions Riesz} is satisfied for
$\alpha \geq 2 \geq \beta$ with $\alpha > \beta$ and $R_*$ the unique solution to
\begin{equation}\label{Eq:EFLJ R eqn}
    R^{\alpha + s}-\eta R^{\beta + s} = \frac{c_{s,d}}{2\gamma}.
\end{equation}

\item \textbf{Exponential:} 
$V(x) =  \frac{\gamma}{\alpha \beta} \exp \left( \alpha \| x \|^{\beta} \right)$,
where $\gamma > 0$, so that
\begin{equation*}
  v(\rho) = \frac{\gamma}{\alpha \beta} \exp \left( \alpha \rho^{\frac{\beta}{2}} \right).
\end{equation*}
Theorem~\ref{thm:Sufficient Conditions Riesz} is satisfied for $\alpha > 0$ and $\beta \geq 2$ and $R_*$ the unique solution to
\begin{equation*}
         R^{\beta + s} \exp \left( \alpha R^{\beta} \right) = \frac{c_{s,d}}{2\gamma}.
\end{equation*}

\item \textbf{Power law with logarithm: }
$V(x) =  \gamma \| x \|^{\alpha}\log \left( \| x \|^2 \right)$, where $\gamma > 0$, so that
\begin{equation*}
    v(\rho) = \gamma \rho^{\frac{\alpha}{2}} \log(\rho).
\end{equation*}
Theorem~\ref{thm:Sufficient Conditions Riesz} is satisfied for $\alpha \geq 2$ and $R_*$ the unique solution to
\begin{equation*}
   R^{\alpha + s} (1+\alpha \log (R)) = \frac{c_{s,d}}{4\gamma}.
\end{equation*}

\item \textbf{Power-law 
with a sink:} $V(x) =  \frac{\gamma}{\alpha} \left| \| x \|^2 - R_0^2 \right|^{\alpha/2}$, where $\gamma, R_0 > 0$, so that
\begin{equation*}
  v(\rho) = \frac{\gamma}{\alpha} | \rho - R_0^2 |^{\alpha/2}.
\end{equation*}
Theorem~\ref{thm:Sufficient Conditions Riesz} is satisfied for $\alpha \geq 2$ and $R_* > R_0$ the unique solution to
\begin{equation*}
    R^{s+2} \left( R^2-R_0^2 \right)^{\frac{\alpha}{2}-1} = \frac{c_{s,d}}{2\gamma}.
\end{equation*}   
\end{itemize}

In contrast to these examples,
Lemmas \ref{lem:General Inside the Sphere}, \ref{lem:General Outside the Sphere} and Appendix~\ref{Ap:Ex not Suff}
give some sufficient conditions that can be used when $v$ is not convex.
Appendix~\ref{Sec:LJEF} discusses an example of a Lennard\,--\,Jones type external field
with $0 > \alpha > \beta$,
for which $v(\rho)$ has a finite limit as $\rho \to \infty$ and hence $v$ is not convex,
but $\sigma_{R_*}$ is still the equilibrium measure.

\subsection{Connection to other works}

One goal of this article is to provide more insight into the following question introduced in \cite{ChaSW23, ChaSW22}: when does dimension reduction of the support of the equilibrium measure occur for Riesz energies with external fields? For 
$V(x) = c \|x\|^{\alpha}$, Theorem \ref{thm:Sphere Min} provides a characterization of when the support of the equilibrium measure is a sphere, which leaves open the question of whether dimension reduction occurs for other combinations of $-2<s<d$ and $\alpha>\max\{0,-s\}$. Thus far, the answer appears to be negative, with  \cite[Theorem 1.2]{ChaSW23} showing that for $s = d-4$ and $ \alpha < \alpha_{d-4,d}^* = 2$, $\mueq$ has a full dimensional component, and with \cite[Theorem II.5]{Bil18}, \cite[Theorem 1.4]{ChaSW22}, \cite[Theorem 17]{HerGBS24}, \cite[Proposition 2.13]{Lop10}, \cite[Theorem 3.2, Example 3.2]{MhasS92}, and \cite{AgaDKKMMS19} (which studies the one-dimensional setting) finding that for certain values of $\alpha$ and $s \geq d-3$, the support of the equilibrium measure is $d$-dimensional and connected, providing more information than our general Theorem \ref{thm:Sphere not a minimizer} in these specific cases. While this work focuses on when the uniform measure on a sphere minimizes the energy, if the support instead has interior points, the distribution of the equilibrium measure on the interior is closely related to the fractional Laplacian of the external field $V$, see e.g.\ \cite{DydKK17a, Kwa17}.

The equilibrium measures for $I_{s,V}$ for power-law external fields, $V(x) = \gamma \frac{\|x\|^{\alpha}}{\alpha} $, appear in other contexts.
For instance, for $s=0$, $d\geq1$, and $\alpha = 2$,
the equilibrium measure is the asymptotic spectral distribution of the vectors of eigenvalues of $d$-tuples of commuting Hermitian $n \times n$ random matrices, as shown in \cite{McC23}. Thus, the equilibrium measures obtained by Chafa\"{\i}, Saff, and Womersley in \cite [Theorem 1.4]{ChaSW22} and \cite [Theorem 1.2(i)-(b)] {ChaSW23} extend to higher dimensions the Wigner semicircular law ($d=1$, Gaussian Unitary Ensemble) and
the circular law ($d=2$, for the planar Coulomb gas of the Complex Ginibre Ensemble). 

Wasserstein gradient flows and steepest descent flows for nonsingular Riesz energies with external fields arising from the maximum mean discrepancy functional have been studied in \cite{AltHS23, HagHABCS23, HerBGS23, HerGBS24, HerWAH23}, in relation to the halftoning problem in image processing. In this setting, the external field is the negative of the potential  of some probability measure, i.e. $V(x) = -U_s^{\nu}(x)$, acting as an attractive sink for mass.  
The Wasserstein steepest descent flows of the Riesz energy functional $I_s$ for $-2 < s < 0$ and an initial measure $\mu_0 = \delta_0$ was also considered in \cite{HerGBS24}, where it was emphasized that determining the direction of steepest descent requires one to solve the following constrained optimization problem:
\begin{equation}\label{eq:Constrained optimization alpha problem}
    \underset{\nu\in \mathcal{P}(\mathbb{R}^d)}{\Minz}\,\, I_{s}(\nu) \quad \text{such that}\quad \int_{\mathbb{R}^d} \| x \|^{\alpha} \rd\nu(x) = 1. 
\end{equation}
It was further shown, in the case that $\alpha = 2$, that this is equivalent to determining the initial step of a minimizing moment scheme (Jordan-Kinderlehrer-Otto scheme):
\begin{equation}\label{eq:JKO scheme}
    \underset{\mu\in \mathcal{P}(\mathbb{R}^d)}{\Minz}\,\, I_{s}(\mu) + \frac{\gamma}{\alpha} W_{\alpha}^{\alpha}(\delta_0,\mu),
\end{equation}
where $W_{\alpha}$ is the \textit{Wasserstein or Kantorovich distance} of order $\alpha$, defined in \eqref{eq:Wasserstein def}.

For $\alpha > 0$, let $\mathcal{P}_{\alpha}(\mathbb{R}^d)$ be the set of probability measures on $\mathbb{R}^d$ with finite $\alpha$-moment, i.e.\ $\int \| x\|^{\alpha} \rd\mu(x)<\infty$, and for $c > 0$, let $\mathcal{P}_{\alpha, c}(\mathbb{R}^d)$ be the set of probability measures with $\int \| x\|^{\alpha} \rd\mu(x) = c^{\alpha}$. For $\alpha \geq 1$, the Wasserstein distance of order $\alpha$, denoted $W_{\alpha}:\mathcal{P}_{\alpha}(\mathbb{R}^d)\times \mathcal{P}_{\alpha}(\mathbb{R}^d) \rightarrow [0,\infty)$, is defined by
\begin{equation}\label{eq:Wasserstein def}
    W_{\alpha}(\mu,\nu) := \left( \min_{\omega \in \Pi(\mu,\nu)} \int_{\mathbb{R}^d\times \mathbb{R}^d} \| x-y \|^{\alpha} \rd\omega(x,y) \right)^{1/\alpha}, \quad \mu,\nu \in \mathcal{P}_{\alpha}(\mathbb{R}^d),
\end{equation}
where $\Pi(\mu,\nu)$ is the set of probability measures on the product space $\mathbb{R}^d\times\mathbb{R}^d$ with marginal distributions $\mu$ and $\nu$. This set is convex and nonempty, since it contains the tensor product $\mu\otimes\nu$. The definition of $W_{\alpha}$ can be extended to $0 < \alpha < 1$.

The constrained optimization problem \eqref{eq:JKO scheme} can be interpreted as an energy minimization problem for a Riesz kernel with a power-law external field, since 
\begin{equation*}
    W_{\alpha}^{\alpha}(\delta_0,\mu) = \int_{\mathbb{R}^d} \| x \|^{\alpha} \rd\mu(x).
\end{equation*}

In the following proposition, we generalize  \cite[Proposition 16]{HerGBS24}, showing the equivalence of \eqref{eq:Constrained optimization alpha problem} and \eqref{eq:JKO scheme} for a wider range of $s$ and $\alpha$.
Its proof can be found in Section \ref{subsec:Proof for Wasserstein Result}.

\begin{proposition}\label{prop:connect to opt control}
Consider the external field $V(x) = \frac{\gamma}{\alpha}\| x \|^{\alpha}$ with $\gamma>0$ and $\alpha > \max\{-s, 0\}$. 
Suppose that $s\in(-\infty,d)$, $s \neq 0$, and denote by $f_{\#}\mu$ the pushforward of $\mu$ by a map $f$.
Then
\begin{itemize}

\item The energy $I_{s}$ is minimized over $\mathcal{P}_{\alpha, 1}(\mathbb{R}^d)$ by $\nu$\\
if and only if $\nu = (c\,\mathrm{Id})_{\#}\mu$ minimizes $I_{s,V}$ over $\mathcal{P}(\mathbb{R}^d)$, with $c = \big( \frac{s I_{s}(\nu)}{2 \gamma} \big)^{ \frac{1}{s+\alpha}}$.

\item The energy $I_{s, V}$ is minimized over $\mathcal{P}(\mathbb{R}^d)$ by $\mu$\\
if and only if $\nu = (c^{-1}\,\mathrm{Id})_{\#}\mu$ minimizes $I_{s}$ over $\mathcal{P}_{\alpha, 1}(\mathbb{R}^d)$, with $c = \Big( \int_{\mathbb{R}^d} \| x\|^{\alpha} \Big)^{ \frac{1}{\alpha}}$.

\end{itemize}

Suppose instead that $s =0$ and $ \alpha >0$. Then

\begin{itemize}

\item The energy $I_{0}$ is minimized over $\mathcal{P}_{\alpha, 1}(\mathbb{R}^d)$ by $\nu$\\if and only if $\mu = (c\,\mathrm{Id})_{\#}\nu$ minimizes $I_{0, V}$ over $\mathcal{P}(\mathbb{R}^d)$, with $c = \big( \frac{1}{\gamma} \big)^{ \frac{1}{\alpha}}$.

\item The energy $I_{0, V}$ is minimized over $\mathcal{P}(\mathbb{R}^d)$ by $\mu$\\if and only if $\nu = (c^{-1}\,\mathrm{Id})_{\#}\mu$ minimizes $I_{0}$ over $\mathcal{P}_{\alpha, 1}(\mathbb{R}^d)$, with $c = \Big( \int_{\mathbb{R}^d} \| x\|^{\alpha} \Big)^{ \frac{1}{\alpha}}$.

\end{itemize}

\end{proposition}

The problem of finding the equilibrium of $I_{s,V}$ for power-law external fields, such as considered in Theorem \ref{thm:Sphere Min}, has a natural analogue, when replacing the Riesz kernel by
\begin{equation*}
    L_{\alpha, \beta}(x-y) = \frac{\|x-y\|^{\alpha}}{\alpha} - \frac{\|x-y\|^{\beta}}{\beta}, \qquad \alpha > \beta > -d
\end{equation*}
\emph{without} the presence of an external field. This Lennard\,--\,Jones type kernel forces particles to repel each other at short range and attract each other when far apart. The minimization of these power-law energies has been the subject of much study, see e.g.\ \cite{BalCLR13, CarDM16, CarFP17,  Fra22, DavLM22, GutCO22, GutCO23, DavLM23, CarS23, FraM24}. In particular, the results of \cite[Theorem 1, Remarks 1 and 2]{BalCLR13} and \cite[Theorems 3.4 and 3.10]{CarDM16} imply that for $-d < \beta \leq 3-d$ the support of any equilibrium measure cannot have finite $(d-1)$-dimensional Hausdorff measure, much like our Theorem \ref{thm:Sphere not a minimizer}. Furthermore, \cite[Theorem 1]{FraM24} shows that if the repulsion is sufficiently weak and the attraction sufficiently strong, $\sigma_R$ is the equilibrium measure of the power-law energy, much like Theorem \ref{thm:Sphere Min}.

\section{Proofs}\label{sec:Proofs}


In what follows, we will always assume that the external field $V$ is a radial function, i.e. there is some function $v$ as in \eqref{eq:v}. Since $K_{s}$ and $V$ are both rotationally invariant, 
%
and since $\mueq$ is unique if it exists, the equilibrium measure must be radial, and there is some $\nu \in \mathcal{P}([0, \infty))$ such that, writing $x = r \theta$ with $r \in [0, \infty)$ and $\theta \in \mathbb{S}^{d-1}$,
\begin{equation*}
    \rd\mueq(x) = \rd \sigma( \theta)\, \rd\nu(r).
\end{equation*}
From Proposition \ref{prop: Riesz/Log Potential of a Sphere} of the Appendix, the modified potential then becomes
\begin{equation*}
    U_{s}^{\mueq}(x) + V(x) =
    \int_{0}^{\infty} r^{-s}
    \hsd{\frac{\|x\|^2}{r^2}}
    \, \rd\nu(r) +
    v(\|x\|^2),
\end{equation*}
where
\begin{equation}\label{eq:Riesz Sphere Potential Function no Radius}
\hsd{\lambda} := 
\begin{cases}
   \frac{1}{s} (1 + \sqrt{\lambda})^{-s} \;
   \HG21\left(\tfrac{s}{2}, \tfrac{d-1}{2}; d-1; \frac{4 \sqrt{\lambda}}{(1+ \sqrt{\lambda})^2}\right) & s \neq 0, \\
    - \log(1 + \sqrt{\lambda}) + \frac{\sqrt{\lambda}}{(1+ \sqrt{\lambda})^2} \;\HG32\left(1, 1,  \tfrac{d+1}{2}; 2, d; \frac{4 \sqrt{\lambda}}{(1+ \sqrt{\lambda})^2}\right) & s = 0
\end{cases}.
\end{equation}
To prove that $\mueq$ is indeed the equilibrium measure for the energy, we need only show that the Frostman conditions \eqref{cond: Frostman} are satisfied.
The radial symmetry of both $\mueq$ and the modified potential then turn this into a one-dimensional problem. In particular, if we want to show that $\sigma_R$ is the equilibrium measure, we need only show that $R^{-s} \hsd{\frac{\|x\|^2}{R^2}} + v(\|x\|^2)$ achieves its global minimum at $\|x\| = R$. For ease of computation and notation, we set 
\begin{equation*}
    \lambda := \frac{\|x\|^2}{R^2}
\end{equation*}
and check that $R^{-s} \hsd{\lambda} + v(\lambda R^2)$ has a global minimum at $\lambda = 1$.

In Section \ref{subsec:Notation}, we introduce the notation that will facilitate the computation of derivatives that are needed in our proofs. In Section~\ref{subsec:Necessary Conditions} we prove
Theorem~\ref{thm:Necessary Condition Riesz}
giving necessary conditions for the uniform measure on a sphere to minimize the energy when $-2 < s < d-3$ and $v$ is sufficiently smooth. We prove Theorem \ref{thm:Sphere not a minimizer} in Section \ref{subsec:Sphere not a minimizer}, showing that we should not expect the uniform measure on a sphere to minimize the energy when $s \geq d-3$. Section~\ref{subsec:sufficient conditions} then provides a variety of sufficient conditions so that the equilibrium measure has spherical support. Each lemma in that section considers the behavior of the modified potential $U_{s}^{\sigma_R}(x) + V(x)$ for specific values of $s$, either outside or inside the sphere $\mathbb{S}_R^{d-1}$, and provides conditions on $V$ that guarantee the modified potential achieves its global minimum on the sphere. In Sections~\ref{subsec:sufficient conditions}
and \ref{subsec:Proof for power law external field} we use some of these sufficiency results to prove Theorems~\ref{thm:Sufficient Conditions Riesz} and \ref{thm:Sphere Min}, respectively. Throughout the paper, increasing means \textit{nondecreasing} and decreasing means \textit{nonincreasing}.

\subsection{Notation for proofs}\label{subsec:Notation}

For $-2 < s <d-1$ and $R > 0$, we use Proposition \ref{prop: Riesz/Log Potential of a Sphere} to rewrite the modified potential of $\sigma_R$ as
\begin{equation*}
 U^{\sigma_R}_s(x) + V(x) =
 \begin{cases}
    R^{-s} \hsd{\frac{\| x \|^2}{R^2}} + 
    v\left( \frac{\| x \|^2}{R^2} \cdot R^2\right) & s \neq 0\\[2ex]
    -\log(R) + \hzd{\frac{\| x \|^2}{R^2}} +
    v\left( \frac{\| x \|^2}{R^2} \cdot R^2\right) & s = 0
    \end{cases},
        \end{equation*}
where $h_{s,d}$ is as in \eqref{eq:Riesz Sphere Potential Function no Radius}. This characterization is useful due to the structure of the derivative of $h_{s,d}$:
\begin{equation}\label{eq:Riesz Sphere Pot Derive, no Radius, Main}
     h_{s,d}'(\lambda)  = \begin{cases}
        - \frac{d-s-2}{2d} \; \HG21 \left( 1+ \frac{s}{2}, \frac{4+s-d}{2}; \frac{d+2}{2}; \lambda \right) & \lambda \leq 1 \\
        - \frac{1}{2}\lambda^{-\frac{s}{2}-1} \; \HG21 \left( 1+ \frac{s}{2}, \frac{2+s-d}{2}; \frac{d}{2}; \lambda^{-1} \right) & \lambda \geq 1 
    \end{cases},
\end{equation}
which holds for $-2 < s < d-2$.
See Appendix~\ref{sec:Appendix Pot of Sphere} for more details. 

The modified potential for $\sigma_R$ is,
as a function of $\lambda \geq 0$, 
\begin{equation}\label{eq:Modified_potential_f_expression}
   f(\lambda) = f_{s,d,R,v}(\lambda) :=\begin{cases}
        R^{-s} \hsd{\lambda} + v(R^2 \lambda) & s \neq 0 \\
        - \log(R) + h_{0,d}(\lambda) + v(R^2 \lambda) & s = 0
    \end{cases},
\end{equation}
so that 
\begin{equation*}
f \left(\frac{\| x \|^2}{R^2}\right) = U^{\sigma_R}_s(x) + V(x).
\end{equation*}
Using \eqref{eq:Riesz Sphere Pot Derive, no Radius, Main} and \eqref{eq:HGderiv} from Appendix \ref{se:appendix HG},  we know that for $\ell \geq 1$ and $\lambda \in [0,1)$
\begin{equation}\label{eq:Derivatives of f_s,d,R,v in sphere}
\begin{split}
    f^{(\ell)}(\lambda) & = R^{-s} h_{s,d}^{(\ell)}(\lambda) + R^{2 \ell} v^{(\ell)}(R^2 \lambda) \\
    &= R^{-s} \frac{2+s-d}{2^{\ell} d} \left(\prod_{j=1}^{\ell - 1} \frac{(2 + 2 j +s - d) (s+2 j)}{d+2j} \right)  \; \HG21 \Big( \frac{s}{2} + \ell, \frac{2 + s -d}{2}+ \ell; \frac{d}{2}+ \ell; \lambda \Big)\\
    & \qquad + R^{2\ell} v^{(\ell)}(R^2 \lambda).
    \end{split}
\end{equation}
Similarly, using \eqref{eq:Riesz Sphere Pot Derive, no Radius, Main} and \eqref{eq:HGderivFor z inverse},  we know that for $\ell \geq 1$ and $\lambda \in (1,\infty)$
\begin{equation}\label{eq:Derivatives_of_f_s_d_R_v_outside_sphere}
\begin{split}
    f^{(\ell)}(\lambda) & = R^{-s} h_{s,d}^{(\ell)}(\lambda) + R^{2 \ell} v^{(\ell)}(R^2 \lambda) \\
    &= R^{-s} \frac{(-1)^{\ell}}{2^{\ell}} \lambda^{-\frac{s}{2}-\ell} \left(\prod_{j=1}^{\ell - 1} (s+2 j) \right) \; \HG21 \Big( \frac{s}{2} + \ell, \frac{2 + s -d}{2}; \frac{d}{2}; \lambda^{-1} \Big) + R^{2\ell} v^{(\ell)}(R^2 \lambda).
\end{split}
\end{equation}
When $\ell < d-s-1$ the limit exists and we can set
\begin{equation*}
     f^{(\ell)}(1) := \lim_{\lambda \rightarrow 1}  f^{(\ell)}(\lambda).
\end{equation*}
In many of our lemmas below, we will be restricting to $[0,1]$ or $[1, \infty)$. In those cases, we will replace the limit with the appropriate one-sided limit.

For $\kappa \in [0,1)$, let
\begin{equation}\label{eq:Def of inverse external field}
     q(\kappa) := \begin{cases}
         2 R^{s+2}\kappa^{- \frac{s}{2}-1} v'(R^2 \kappa^{-1}) & \kappa > 0 \\
         \lim\limits_{t \rightarrow 0^+} 2 R^{s+2} t^{- \frac{s}{2}-1} v'(R^2 t^{-1}) & \kappa = 0.
     \end{cases}
 \end{equation} 
 and
 \begin{equation}\label{eq: y_s,d def}
 y_{s,d}(\kappa) := -\HG21 \Big( \frac{s}{2} + 1, \frac{2 + s -d}{2}; \frac{d}{2}; \kappa \Big)
 \end{equation}
and let
\begin{equation}\label{eq:def of g_s,d,q}
g(\kappa) = g_{s,d,q}(\kappa) := y_{s,d}(\kappa) + q(\kappa),
\end{equation}
so
\begin{equation}\label{eq:equivalence with f and q}
    g(\kappa) =  2 R^s \kappa^{- \frac{s}{2}-1} f'(\kappa^{-1}).
\end{equation}
For all $\ell \geq 0$ and $\kappa \in [0,1)$, \eqref{eq:HGderiv} yields
\begin{align}\label{eq:Derivatives of g_s,d,q}
    g^{(\ell)}(\kappa) = & \ 
    y_{s,d}^{(\ell)}(\kappa) + q^{(\ell)}(\kappa) \nonumber\\ 
     = &  - \frac{1}{2^{\ell}} \left( \prod_{j=0}^{\ell - 1} \frac{(2 + 2 j +s - d) (s+2 j+2)}{d+2j} \right) \times \nonumber \\
     &  \HG21 \Big( \frac{s}{2} + \ell+1, \frac{2 + s -d}{2}+ \ell; \frac{d}{2}+ \ell; \kappa \Big)
     + q^{(\ell)}(\kappa),
\end{align}
and when $\ell < d -s - 2 $ the limit exists and we set
\begin{equation*}
g^{(\ell)}(1) := \lim\limits_{\kappa \rightarrow 1^-} g^{(\ell)}(\kappa).
\end{equation*}

\subsection{Proof of Theorem~\ref{thm:Necessary Condition Riesz} (necessary conditions)}\label{subsec:Necessary Conditions}


When $\sigma_R$ is the equilibrium measure,
the Frostman conditions \eqref{cond: Frostman}
imply that $\lambda=1$ must be a global minimizer of $f$ on $[0,\infty)$, see the discussion at the start of Section \ref{sec:Proofs}.
For $s < d - 3$, $f$ is twice continuously differentiable, so
the following four conditions must be satisfied:
$f'(1)=0$, $f''(1)\geq 0$, $f(0)\geq f(1)$, and
$\lim\limits_{\lambda \rightarrow \infty} f(\lambda) \geq f(1)$.

From \eqref{eq:Riesz Sphere Pot Derive, no Radius, Main}, \eqref{eq:HGderiv}, \eqref{eq:HG1pos}, \eqref{eq:Equating Higher Geom Series}, and \eqref{eq:c_s,d definition},  we have
\begin{align*}
    h_{s,d}'(1) & = \frac{s-d+2}{2d} \; \HG21 \left( \frac{s}{2}+1,\frac{s-d+2}{2}+1;\frac{d}{2}+1;1 \right) \\
    & = -\frac{1}{4} \; \HG21 \left( \frac{s}{2},\frac{s-d+2}{2};\frac{d}{2};1 \right) = -\frac{c_{s,d}}{4},
\end{align*}
and
\begin{align*}
    \hsdpp{1} &= \frac{(s-d+2)(s+2)(s-d+4)}{4d(d+2)}\; \HG21 \left( \frac{s}{2}+2,\frac{s-d+2}{2}+2;\frac{d}{2}+2;1 \right)\\
    &= \frac{(s+2)(d-s-4)}{16(d-s-3)} \; \HG21 \left( \frac{s}{2},\frac{s-d+2}{2};\frac{d}{2};1 \right) = \frac{(s+2)(d-s-4)c_{s,d}}{16(d-s-3)}.
\end{align*}
For the first condition $f'(1) = 0$, we have
\begin{equation*}
            0 = R^{-s}h_{s,d}'(1) + R^2 v'(R^2) =
            -\frac{c_{s,d}}{4} R^{-s} + R^2 v'(R^2),
\end{equation*}
which implies \ref{item:NecCond SP}.

For the second condition $f''(1) \geq 0$, we have
\begin{equation*}
    0\leq R^{-s} \hsdpp{1} + R^4 v''(R^2) =
    \frac{(s+2)(d-s-4)}{4(d-s-3)} \frac{c_{s,d}}{4} R^{-s} + R^4 v''(R^2).
\end{equation*}
Then for $R > 0$, $v'(R^2) = \frac{c_{s,d}}{4} R^{-s-2} > 0$ gives \ref{item:NecCond DD}.
        
If $s \neq 0$, $f(0) \geq f(1)$ is equivalent to
\begin{equation*}
    R^{-s}h_{s,d}(0)+v(0)\geq R^{-s}h_{s,d}(1)+v(R^2).
\end{equation*}
Note that by \eqref{eq:c_s,d definition}, \eqref{eq:FHHGb}, and \eqref{eq:Riesz Sphere Potential Function no Radius}, $h_{s,d}(0)=1/s$ and $h_{s,d}(1)=c_{s,d}/s$.
On the other hand, if $s = 0$, $f(0) \geq f(1)$ is equivalent to
\begin{equation*}
    h_{0,d}(0)+v(0)\geq h_{0,d}(1)+v(R^2).
\end{equation*}
Again, by \eqref{eq:b_d definition} and  \eqref{eq:Riesz Sphere Potential Function no Radius},
$h_{0,d}(0) = 0$ and $h_{0,d}(1) = b_d$.
We thus obtain \ref{item:NecCond BC0}. 

If $s \neq 0$, then \eqref{eq:c_s,d definition}, \eqref{eq:Riesz Sphere Potential Function no Radius}, and \eqref{eq:HG} show us that our last condition on $f$ is equivalent to
\begin{align*}
    \frac{R^{-s}}{s} c_{s,d} + v(R^2) &
    \leq \lim_{\lambda \rightarrow \infty} \left( v(R^2 \lambda) + \frac{R^{-s}}{s} (1 + \sqrt{\lambda})^{-s} \left( 1 + \frac{s \sqrt{\lambda}}{(1 + \sqrt{\lambda})^2} + \mathcal{O}\Bigl( \frac{1}{\lambda}\Bigr) \right) \right)\\
    & = \lim_{r \rightarrow \infty} \left( v(r^2) + \frac{(R + r)^{-s}}{s} \right).
\end{align*}
If $s=0$, due to \eqref{eq:b_d definition}, \eqref{eq:Riesz Sphere Potential Function no Radius}, and \eqref{eq:HG32}, our last condition is instead equivalent to
\begin{align*}
    - \log(R) + b_d + v(R^2) & \leq \lim_{\lambda \rightarrow \infty} \Big( -\log(R) - \log(1 + \sqrt{\lambda}) + \mathcal{O}(\frac{1}{\sqrt{\lambda}}) + v( R^2 \lambda) \Big)\\
    & = \lim_{r \rightarrow \infty} \Big( v(r^2) - \log(R + r) \Big).
\end{align*}
We now obtain \ref{item:NecCond BCinf}.  \qedsymbol{}


\begin{remark}[Alternative viewpoint]\label{rem: Eng SP}
Assuming that $\sigma_R$ is the equilibrium measure, the energy \eqref{eq:Energy Def} of $\sigma_R$,
using Proposition~\ref{prop:Riesz Energy of a sphere}, is
\[
  I_{s,V}(\sigma_R) = 2 v(R^2) +
  \begin{cases}
     \frac{R^{-s}}{s} c_{s,d}, & s \neq 0\\
     - \log(R) + b_{d}, & s = 0
  \end{cases} .
\]
Treating this as a function of $R$, the derivative (remembering that $c_{0, d} = 1$)
\[
  \frac{\partial }{\partial R} \, I_{s,V}(\sigma_R) = 4 R v'(R^2) - 
     c_{s,d} R^{-s-1} 
\]
shows that the necessary condition \ref{item:NecCond SP} must be satisfied at a stationary point.
Moreover,
\[
  \frac{\partial^2 }{\partial R^2} \, I_{s,V}(\sigma_R) = 4 v'(R^2) + 8 R^2 v''(R^2) +
     (s+1) c_{s,d} R^{-s-2} \geq 0,
\]
at a minimum, gives $\frac{R^2 v''(R^2)}{ v'(R^2)} \geq -\frac{s+2}{2}$ at a stationary point,
which is weaker than \ref{item:NecCond DD}.
\end{remark}

\subsection{Proof of Theorem \ref{thm:Sphere not a minimizer}}\label{subsec:Sphere not a minimizer}

In order to prove Theorem \ref{thm:Sphere not a minimizer}, we must consider three cases separately. For the first, we will need the following well-known result (see, e.g., \cite[Theorem 4.3.1]{Fal03}).
\begin{theorem}\label{thm:Dim of support for Capacity}
Let $A \subset \mathbb{R}^d$ be compact and $s > 0$. If the $s$-dimensional Hausdorff measure of $A$, $\mathcal{H}_s(A)$, is finite then for every probability measure $\mu \in \mathcal{P}(A)$,
\[
  \int_A \int_A  K_{s}(x,y) \rd\mu(x) \rd\mu(y) = \infty.
\]
\end{theorem}
In other word, if there exists some measure $\mu \in \mathcal{P}(A)$ such that
$$ \int_A \int_A  K_{s}(x,y) \rd\mu(x) \rd\mu(y) < \infty,$$
then either $\dim_H(A) > s$, or $\dim_H(A) = s$ and $\mathcal{H}_s(A) = \infty$, where $\dim_H(A)$ is the Hausdorff dimension of $A$.

\begin{lemma}\label{lem:perfect set as a minimizer}
Suppose $I_{s,V}$ has a unique compactly supported equilibrium measure $\mueq$ and one of the following is true:
\begin{enumerate}[label=\textnormal{(\alph*)}]
    \item\label{item:Sphere not Min i}  $d-1 \leq s < d$ and $v$ is lower semi-continuous and bounded from below,
    \item\label{item:Sphere not Min ii} $d-2 \leq s < d-1$ and $v$ is bounded from below and $\mathcal{C}^1$ in the extended sense,
    \item\label{item:Sphere not Min iii} $d \geq 3$, $d-3 \leq s < d-2$, and $v$ is bounded from below and $\mathcal{C}^2$ in the extended sense such that $v''$ is finite on $(0, \infty)$.
    \item\label{item:Sphere not Min iv} $d = 2$, $-1 \leq s < 0$, and $v$ is bounded from below, $\mathcal{C}^2$ in the extended sense, and such that $\lim\limits_{\rho \rightarrow 0^+} \rho^{\frac{s}{2}+1} v'(\rho) = 0$ and $v''$ is finite on $(0, \infty)$.
\end{enumerate}
With $x = r \theta$, $r \in [0, \infty)$, and $\theta \in \mathbb{S}^{d-1}$, let $\nu \in \mathcal{P}\big([0,\infty)\big)$ such that
\begin{equation*}
\rd\mueq(x) = \rd\sigma(\theta)\rd\nu(r).
\end{equation*}
Then $\supp(\nu)$ must be a perfect set (i.e. have no isolated points).
\end{lemma}

\begin{proof}[Proof of Theorem \ref{thm:Sphere not a minimizer}]
Our claim follows immediately from Lemma \ref{lem:perfect set as a minimizer}. 
\end{proof}

\textit{Overview of the proof of Lemma \ref{lem:perfect set as a minimizer}.} We need to rule out that there is any isolated point in the support of $\nu$. We first address the possibility that there is an isolated point mass at $0$, using that fact that $K_s$ is singular at $0$ in the first three cases,
\ref{item:Sphere not Min i}-\ref{item:Sphere not Min iii},
and the fact that the modified potential is strictly decreasing near $0$
in the last case \ref{item:Sphere not Min iv}.
We then consider the possibility of an isolated sphere $\mathbb{S}_R^{d-1}$ in the support of $\mueq$, and find contradictions in each setting. 

In case \ref{item:Sphere not Min i},
Theorem \ref{thm:Dim of support for Capacity} immediately tells us that this contributes an infinite amount of energy to our total energy, contradicting that we have finite energy. In the remaining three cases, we consider the potential $U_{s}^{\mueq}$, which may have contributions from the support from the interior of the sphere, the support from the exterior of the sphere, and the support on the sphere itself. 

To handle case \ref{item:Sphere not Min ii},
we write the modified potential as a one-dimensional function and take a derivative. We then find that as $\|x\|$ approaches $R$, the contributions to the derivative from any support inside or outside the sphere are finite, but the contribution from the sphere itself goes to $\infty$ as $\|x\|$ approaches $R$ from below, and $- \infty$ as $\|x\|$ approaches $R$ from above. Since, by the Frostman conditions \eqref{cond: Frostman}, the modified potential must achieve its minimum at $\|x\| = R$, this means that $v'(\|x\|^2)$ must approach $- \infty$ and $\infty$ as $\|x\|$ approaches $R$ from below and above, respectively. This breaks the continuous differentiability of $v$. 

For cases \ref{item:Sphere not Min iii} and \ref{item:Sphere not Min iv},
we also interpret the modified potential as a one-dimensional function and utilize the second derivative. We then see that as $\|x\|$ approaches $R$, the contributions from any support inside or outside the sphere are finite, but the contribution from the sphere itself goes to $-\infty$ as $\|x\|$ approaches $R$ from below. Since, by the Frostman conditions \eqref{cond: Frostman}, the modified potential must achieve its minimum at $\|x\| = R$, this means that $v''(\|x\|^2)$ must approach $\infty$ as $\|x\|$ approaches $R$ from below. This breaks the finiteness and continuity of the second derivative  of $v$.

\begin{proof}[Proof of Lemma \ref{lem:perfect set as a minimizer}]

Let $m_0 := \nu(\{0\})$. Suppose, for the sake of contradiction, that $s \geq 0$ and $m_0 > 0$. Since $v$ is bounded from below by some constant $C$
$$ I_{s,V}(\mueq) \geq 2 C + m_0^2 I_s(\delta_0) = \infty  $$
which contradicts $I_{s,V}(\mueq)$ being finite, and so $m_0 = 0$. This means that $0$ cannot be an isolated point in the support of $\nu$.

In case \ref{item:Sphere not Min iv}, $K_s$ is not singular, so we may have $m_0 > 0$. Suppose, for the sake of contradiction, that $0$ is an isolated point in the support of $v$. Then our modified potential is
\[
  m_0 \frac{\|x\|^{-s}}{s} + v(\|x\|^2) + 
  \int_{r_2}^{\infty} r^{-s} \hsd{\frac{\|x\|^2}{r^2}} \rd\nu(r),
\]
for some radius $r_2>0$. We can rewrite this as
\[
  q(\rho) = m_0 \frac{\rho^{-\frac{s}{2}}}{s} + v(\rho) + 
  \int_{r_2}^{\infty} r^{-s} \hsd{\frac{\rho}{r^2}} \rd\nu(r)
\]
so for $\rho \in (0, r_2^2),$
\[
  q'(\rho) = -\frac{m_0}{2} \rho^{-\frac{s}{2}-1} + v'(\rho) + 
  \int_{r_2}^{\infty} r^{-s-2} \hsdp{\frac{\rho}{r^2}} \rd\nu(r).
\]
From \eqref{eq:positivity of 2F1 on [0,1)} and \eqref{eq:HGderiv}, we have that $ \HG21 \Big( \frac{s}{2} + 1, \frac{2 + s -d}{2} +1; \frac{d}{2}+1; t \Big)$ is positive and increasing on $[0,1)$, so for $0 \leq \rho < r_2^2$, we have that
\begin{align*}
 0 & \leq \int_{r_2}^{\infty} r^{-s-2}  \; \HG21 \Big( \frac{s}{2} + 1, \frac{2 + s -d}{2}+ 1; \frac{d}{2}+ 1;  \frac{\rho}{r^2} \Big) \rd\nu(r)\\
& \leq \nu([r_2, \infty)) r_2^{-s-2}  \; \HG21 \Big( \frac{s}{2} + 1, \frac{2 + s -d}{2}+ 1; \frac{d}{2}+ 1;  \frac{\rho}{r_2^2} \Big)\\
& \leq  r_2^{-s-2} \; \HG21 \Big( \frac{s}{2} + 1, \frac{2 + s -d}{2}+ 1; \frac{d}{2}+ 1;  1 \Big)  < \infty.
\end{align*}

Since $\lim_{\rho \rightarrow 0^+} -\frac{m_0}{2} \rho^{-\frac{s}{2}-1} + v'(\rho) = - \infty$, we see
that $q$ is strictly decreasing near $0$, and so cannot achieve its minimum at $0$,
contradicting the Frostman conditions.
Thus $0$ is not an isolated point of $\supp \nu$.

Now suppose, for the sake of contradiction, that $R>0$ is an isolated point in the support of $\nu$, so $m_R := \nu(\{R\}) > 0$.
Let $0\leq r_1< R < r_2$ be such that $(r_1, r_2) \cap \supp(\nu) = \{ R \}$
and let $\tilde{\nu} = \nu - m_0 \delta_0$.
Then
\begin{equation*}
\begin{split}
  U_{s}^{\mueq}(x) = & m_0 \frac{\|x\|^{-s}}{s} +
  \int_{0}^{r_1}  r^{-s} \hsd{\frac{\|x\|^2}{r^2}} \rd\tilde{\nu}(r)\\ 
  & + m_R R^{-s} \hsd{\frac{\|x\|^2}{R^2}} + 
  \int_{r_2}^{\infty} r^{-s} \hsd{\frac{\|x\|^2}{r^2}} \rd\tilde{\nu}(r).
\end{split}
\end{equation*}
Analogous to \eqref{eq:Modified_potential_f_expression}, 
the modified potential, 
as a function of $\lambda = \frac{\|x\|^2}{R^2}$,
is then
\begin{equation}\label{eq:Defintion of p, mod potential, perfect set}
\begin{split}
  p( \lambda) & := 
  m_0\frac{R^{-s} \lambda^{-\frac{s}{2}}}{s} +
  \int_{0}^{r_1}  r^{-s} \hsd{\frac{R^2}{r^2} \lambda} \rd\tilde{\nu}(r)\\
  & \qquad+m_R R^{-s} \hsd{\lambda} + \int_{r_2}^{\infty} r^{-s} \hsd{\frac{R^2}{r^2} \lambda} \rd\tilde{\nu}(r) + v(R^2 \lambda).
\end{split}
\end{equation}
Recall that in cases \ref{item:Sphere not Min i}, \ref{item:Sphere not Min ii}, and \ref{item:Sphere not Min iii}, $m_0 = 0$ and $\tilde{\nu} = \nu$.

We now need to show a contradiction in all four cases, by showing that $R$ cannot be an isolated point, and proving our claim.

\vspace{1ex}
\underline{Case} \ref{item:Sphere not Min i}: Since adding a constant to the external field does not change the equilibrium measure, we may assume, without loss of generality, that there is a constant $C$ such that $V(x) \geq C$ for all $x \in \mathbb{R}^d$, and $K_{s}(x-y) + 2C \geq 0$ for $x, y \in \supp(\mueq)$. We then have that
$$ \infty > I_{s,V}(\mueq)  \geq 
m^2 ( I_{s}(\sigma_R) +2C) .$$
However, Theorem \ref{thm:Dim of support for Capacity} tells us that $I_s(\sigma_R) = \infty$, which is a contradiction.

\vspace{1ex}
\underline{Case} \ref{item:Sphere not Min ii}: With $p$ as in \eqref{eq:Defintion of p, mod potential, perfect set}, and using \eqref{eq:Derivatives of f_s,d,R,v in sphere} and \eqref{eq:Derivatives_of_f_s_d_R_v_outside_sphere}, we see that for $\lambda \in ( \frac{r_1^2}{R^2}, 1) \cup (1, \frac{r_2^2}{R^2})$,
\begin{align*}
p'( \lambda)  = & \int_{0}^{r_1}  R^2 r^{-s-2} \hsdp{\frac{R^2}{r^2} \lambda}  \rd\nu(r) +  m_R R^{-s} \hsdp{\lambda} \\
&  + \int_{r_2}^{\infty} R^2 r^{-s-2} \hsdp{\lambda \frac{R^2}{r^2}} \rd\nu(r) + R^2 v '(R^2 \lambda)\\
 = & -\frac{1}{2} \int_{0}^{r_1}  R^{-s} \lambda^{-\frac{s}{2}-1} \; \HG21 \Big( \frac{s}{2} + 1, \frac{2 + s -d}{2}; \frac{d}{2}; \frac{r^2}{R^2} \lambda^{-1} \Big) \rd\nu(r)   \\
& + \frac{2+s-d}{2 d} \int_{r_2}^{\infty} R^2 r^{-s-2}  \; \HG21 \Big( \frac{s}{2} + 1, \frac{2 + s -d}{2}+ 1; \frac{d}{2}+ 1;  \frac{R^2}{r^2} \lambda \Big) \rd\nu(r) \\
& +  m_R R^{-s} \hsdp{\lambda}  +  R^2 v '(R^2 \lambda).
\end{align*}
Suppose that $d-2 < s < d-1$. For the first integral, \eqref{eq:positivity of 2F1 on [0,1)} and \eqref{eq:HGderiv} tell us that $ \HG21 \Big( \frac{s}{2} + 1, \frac{2 + s -d}{2}; \frac{d}{2}; t \Big)$ is positive and increasing on $[0,1)$, and combining this with the facts that $\frac{r_1^2}{R^2} < \lambda$ and $r_1 < R$, we have
\begin{align*}
0 & \leq \lim_{\lambda \rightarrow 1}\int_{0}^{r_1}  \lambda^{-\frac{s}{2}-1} \; \HG21 \Big( \frac{s}{2} + 1, \frac{2 + s -d}{2}; \frac{d}{2}; \frac{r^2}{R^2} \lambda^{-1} \Big) \rd\nu(r) \\
& \leq  \lim_{\lambda \rightarrow 1} \nu([0,r_1]) \lambda^{-\frac{s}{2}-1} \; \HG21 \Big( \frac{s}{2} + 1, \frac{2 + s -d}{2}; \frac{d}{2}; \frac{r_1^2}{R^2} \lambda^{-1} \Big)\\
& \leq \HG21 \Big( \frac{s}{2} + 1, \frac{2 + s -d}{2}; \frac{d}{2}; \frac{r_1^2}{R^2} \Big) < \infty.
\end{align*}
For the second integral, \eqref{eq:positivity of 2F1 on [0,1)} and \eqref{eq:HGderiv} tell us that $ \HG21 \Big( \frac{s}{2} + 1, \frac{2 + s -d}{2} +1; \frac{d}{2}+1; t \Big)$ is positive and increasing on $[0,1)$, and since $\lambda < \frac{r_2^2}{R^2}$ and $r_2 > R$, we have
\begin{align*}
 0 & \leq \lim_{\lambda \rightarrow 1} \int_{r_2}^{\infty} r^{-s-2}  \; \HG21 \Big( \frac{s}{2} + 1, \frac{2 + s -d}{2}+ 1; \frac{d}{2}+ 1;  \frac{R^2}{r^2} \lambda \Big) \rd\nu(r)\\
& \leq \lim_{\lambda \rightarrow 1} \nu([r_2, \infty)) r_2^{-s-2}  \; \HG21 \Big( \frac{s}{2} + 1, \frac{2 + s -d}{2}+ 1; \frac{d}{2}+ 1;  \frac{R^2}{r_2^2} \lambda \Big)\\
& \leq  r_2^{-s-2}  \; \HG21 \Big( \frac{s}{2} + 1, \frac{2 + s -d}{2}+ 1; \frac{d}{2}+ 1;  \frac{R^2}{r_2^2} \Big) < \infty.
\end{align*}
Thus, the limits of the two integrals, as $\lambda \rightarrow 1$, are finite. 

From \eqref{eq:Derivatives of f_s,d,R,v in sphere} and \eqref{eq:HG1neg}, we see that  $\lim\limits_{\lambda \rightarrow 1^{+}} \hsdp{\lambda} = - \infty$ and $\lim\limits_{\lambda \rightarrow 1^{-}} \hsdp{\lambda} = \infty$. But we know that $p$ must achieve its minimum at $1$, so we must have that $\lim\limits_{\lambda \rightarrow 1^{+}} p'(\lambda) \geq 0$ and $\lim\limits_{\lambda \rightarrow 1^{-}} p'(\lambda) \leq 0$. This, in turn, then means that $\lim\limits_{\lambda \rightarrow 1^{+}} v'(R^2 \lambda) = \infty$ and $ \lim\limits_{\lambda \rightarrow 1^{-}} v'(R^2 \lambda) = -\infty$, but then $v$ is not continuously differentiable at $1$, a contradiction.

In the case $s = d-2$, we see that
\begin{align*}
p'(\lambda) & = -\frac{1}{2} R^{-s} \int_{0}^{r_1} \lambda^{-\frac{s}{2}-1} \rd\nu(r) + R^{-s} \hsdp{\lambda} + R^2 v'(R^2 \lambda) \\
& = -\frac{1}{2} R^{-s} \nu([0,r_1]) \lambda^{-\frac{s}{2}-1} + R^{-s} \hsdp{\lambda} + R^2 v'(R^2 \lambda).
\end{align*}
Since $p$ must achieve its minimum at $1$, we must have $\lim\limits_{\lambda \rightarrow 1^+} p'(\lambda) \geq 0$ and $\lim\limits_{\lambda \rightarrow 1^-} p'(\lambda) \leq 0$. But then we must have (using \eqref{eq:Derivatives_of_f_s_d_R_v_outside_sphere}) 
$$ \lim_{\lambda \rightarrow 1^+} R^2 v'(R^2 \lambda) \geq \frac{1}{2} R^{-s} \big( m + \nu([0,r_1]) \big)$$
and, due to \eqref{eq:Derivatives of f_s,d,R,v in sphere},
$$ \lim_{\lambda \rightarrow 1^-} R^2 v'(R^2 \lambda) \leq \frac{1}{2} R^{-s}  \nu([0,r_1]).$$
But then $v$ is not continuously differentiable at $1$, giving us a contradiction.

\vspace{1ex}
\underline{Cases} \ref{item:Sphere not Min iii} and \ref{item:Sphere not Min iv}:
Again with $p$ as in \eqref{eq:Defintion of p, mod potential, perfect set} and using  \eqref{eq:Derivatives_of_f_s_d_R_v_outside_sphere} and \eqref{eq:Derivatives of f_s,d,R,v in sphere}, we find that for $\lambda \in ( \frac{r_1^2}{R^2}, 1) \cup (1, \frac{r_2^2}{R^2})$,
\begin{align*}
  p''( \lambda)  = & \frac{m_0}{4} R^{-s} (s+2) \lambda^{- \frac{s}{2}-2} +
  \int_{0}^{r_1}  R^4 r^{-s-4} \hsdpp{\frac{R^2}{r^2} \lambda} \rd\tilde{\nu}(r) \\
  &  +  m_R R^{-s} \hsdpp{\lambda} + \int_{r_2}^{\infty} R^4 r^{-s-4} \hsdpp{\lambda \frac{R^2}{r^2}} \rd\nu(r) + R^4 v ''(R^2 \lambda)\\
 = & \frac{m_0}{4} R^{-s} (s+2) \lambda^{- \frac{s}{2}-2}+  \frac{s+2}{4} \int_{0}^{r_1}  R^{-s} \lambda^{-\frac{s}{2}-2} \; \HG21 \Big( \frac{s}{2} + 2, \frac{2 + s -d}{2}; \frac{d}{2}; \frac{r^2}{R^2} \lambda^{-1} \Big) \rd\tilde{\nu}(r)  \\
&  +  \frac{(2+s-d)(4+s-d)(s+2)}{4 d (d+2)} \int_{r_2}^{\infty} R^4 r^{-s-4}  \; \HG21 \Big( \frac{s}{2} + 2, \frac{2 + s -d}{2}+ 2; \frac{d}{2}+ 2;  \frac{R^2}{r^2} \lambda \Big) \rd\nu(r)\\
& + m_R R^{-s} \hsdpp{\lambda}  +  R^4 v ''(R^2 \lambda).
\end{align*}
For the first integral, from $0$ to $r_1$, by \eqref{eq:positivity of 2F1 on [0,1)} and \eqref{eq:HGderiv}, $ \HG21 \Big( \frac{s}{2} + 2, \frac{2 + s -d}{2}; \frac{d}{2}; t \Big)$ is decreasing on $[0,1)$, so it achieves its maximum at $t = 0$, and since $\frac{r_1^2}{R^2} < \lambda$ and $r_1 < R$, we have
\begin{align*}
\tilde{\nu}([0,r_1]) & \; \HG21 \Big( \frac{s}{2} + 2, \frac{2 + s -d}{2}; \frac{d}{2}; \frac{r_1^2}{R^2} \Big) \\ & =  \lim_{\lambda \rightarrow 1} \tilde{\nu}([0,r_1]) \lambda^{-\frac{s}{2}-2} \; \HG21 \Big( \frac{s}{2} + 2, \frac{2 + s -d}{2}; \frac{d}{2}; \frac{r_1^2}{R^2} \lambda^{-1} \Big)\\ 
& \leq \lim_{\lambda \rightarrow 1}\int_{0}^{r_1}  \lambda^{-\frac{s}{2}-2} \; \HG21 \Big( \frac{s}{2} + 2, \frac{2 + s -d}{2}; \frac{d}{2}; \frac{r^2}{R^2} \lambda^{-1} \Big) \rd\tilde{\nu}(r) \\
& \leq  \lim_{\lambda \rightarrow 1} \tilde{\nu}([0,r_1]) \lambda^{-\frac{s}{2}-2} \\
& = \tilde{\nu}([0,r_1]).
\end{align*}
For the second integral, from $r_2$ to $\infty$, \eqref{eq:positivity of 2F1 on [0,1)} and \eqref{eq:HGderiv} tell us that $ \HG21 \Big( \frac{s}{2} + 2, \frac{2 + s -d}{2} +2; \frac{d}{2}+2; t \Big)$ is positive and increasing on $[0,1)$, and since $\lambda < \frac{r_2^2}{R^2}$ and $r_2 > R$, we have
\begin{align*}
 0 & \leq \lim_{\lambda \rightarrow 1} \int_{r_2}^{\infty} r^{-s-4}  \; \HG21 \Big( \frac{s}{2} + 2, \frac{2 + s -d}{2}+ 2; \frac{d}{2}+ 2;  \frac{R^2}{r^2} \lambda \Big) \rd\nu(r)\\
& \leq \lim_{\lambda \rightarrow 1} \nu([r_2, \infty)) r_2^{-s-4}  \; \HG21 \Big( \frac{s}{2} + 2, \frac{2 + s -d}{2}+ 2; \frac{d}{2}+ 2;  \frac{R^2}{r_2^2} \lambda \Big)\\
& \leq  r_2^{-s-4}  \; \HG21 \Big( \frac{s}{2} + 2, \frac{2 + s -d}{2}+ 2; \frac{d}{2}+ 2;  \frac{R^2}{r_2^2} \Big) < \infty.
\end{align*}
Thus, the limits as $\lambda \rightarrow 1$ of $\frac{m_0}{4}R^{-s} (s+2) \lambda^{-\frac{s}{2}-2}$ and the two integrals, from $0$ to $r_1$ and $r_2$ to $\infty$, are finite.

From \eqref{eq:Derivatives of f_s,d,R,v in sphere} and \eqref{eq:HG1neg} (or \eqref{eq:HG1zer} in the case $s = d-3$), we see that $\lim\limits_{\lambda \rightarrow 1^{-}} \hsdpp{\lambda} = -\infty$. But we know that $p$ must achieve its minimum at $1$, so we must have that $\lim\limits_{\lambda \rightarrow 1^{-}} p''(\lambda) \geq 0$, so  $\lim\limits_{\lambda \rightarrow 1^{-}} v''(R^2 \lambda) = \infty$. But this contradicts $v''$ being continuous and finite on $(0, \infty)$.
\end{proof}

\subsection{Proof of Theorem~\ref{thm:Sufficient Conditions Riesz} and more about sufficient conditions}
\label{subsec:sufficient conditions}

\begin{proof}[Proof of Theorem~\ref{thm:Sufficient Conditions Riesz}]
The result follows immediately from Corollaries \ref{cor:Convex_external_fields_-2_s_d-4_inside_sphere} and \ref{cor:Convex_external_fields_-2_s_d-4_outside_sphere} below, and the Frostman conditions and uniqueness of the equilibrium measure (see the discussion at the start of Section \ref{sec:Proofs}).  
\end{proof}

    For the rest of this section, we provide a useful tool to verify that the function $f:= f_{s,d,R,v}$ from \eqref{eq:Modified_potential_f_expression} indeed achieves its global minimum at $1$, i.e. the uniform measure $\sigma_R$ on the sphere of radius $R$ is the equilibrium measure for $I_{s,V}$. 
    Roughly speaking, when the necessary conditions of Theorem \ref{thm:Necessary Condition Riesz} are satisfied,
    in particular the endpoint behavior at $0$ (see part \ref{item:NecCond BC0}), and $f$ is increasing then decreasing on $[0,1]$
    we deduce that $1$ is the minimizer on $[0, 1]$.
    Similarly on the interval $[1, \infty)$, taking into account Theorem \ref{thm:Necessary Condition Riesz}\ref{item:NecCond BCinf}, if $f$ is increasing then decreasing
    we deduce that $1$ is the minimizer on $[1,\infty)$.
    (This includes the straightforward case when $f$ is decreasing on $[0, 1]$ and increasing on $[1, \infty)$).

    On each of the intervals $[0, 1]$ and $[1, \infty)$
    such functions are unimodal, a term originating in probability and statistics
    (see, for example, \cite{And55,Mud66}).
    More precisely, for an interval $I\subset \mathbb{R}$, a function $\varphi:I\to \mathbb{R}\cup\{ \pm \infty \}$ is \emph{unimodal} on $I$ if there is a point $\xi\in I$ such that $\varphi$ is increasing on $(-\infty,\xi]\cap I$ and decreasing on $[\xi,\infty)\cap I$.
    Note that the degenerate cases when $\xi$ is an end-point of $I$ are treated as unimodal.

    Showing the unimodality of a function is not always easy.
    In Proposition \ref{prop:half-monotone implies unimodal} below, we  give a useful condition to verify unimodality in the proofs of Theorem \ref{thm:Sphere Min} and Theorem~\ref{thm:Sufficient Conditions Riesz}.
    For convenience, we first introduce the following term: for an interval $I\subset \mathbb{R}$ and integers $k_0$, $k$ with $k \geq k_0\geq 1$, a function $\varphi:I\to \mathbb{R}\cup\{ \pm \infty \}$ is called (negatively) \textbf{half-monotone} of order $(k_0,k)$ at a point $a\in I$ if the following conditions hold:
		\begin{enumerate}
			\item $\varphi\in C^k(I)$ in the extended sense and $(-1)^k\varphi^{(k)}\leq 0$ on $I$;
			\item $(-1)^{\ell}\varphi^{(\ell)}(a)\geq 0$ for all $1\leq \ell < k_0$;
			\item $(-1)^{\ell}\varphi^{(\ell)}(a)\leq 0$ for all $k_0 \leq \ell\leq k$.
		\end{enumerate}
		In addition, $\varphi$ is called \textbf{strictly half-monotone} of order $(k_0,k)$ at $a$ if $\varphi^{(k_0-1)}(a)\neq 0$. 
  We assume, without loss of generality, that $k_0$ is the smallest integer such that these properties hold.

        \begin{proposition}[From half-monotone to unimodal]\label{prop:half-monotone implies unimodal}
		Let $k_0$, $k$ be integers with $k \geq k_0\geq 1$. If $\varphi:[0,1]\to \mathbb{R}\cup\{ \pm \infty \}$ is a half-monotone function of order $(k_0,k)$ at $1$, then it is unimodal on $[0,1]$. 
  If, in addition, $k_0=1$, then $\varphi$ is increasing on $[0,1]$. Otherwise, when $\varphi$ is strictly half-monotone at $1$ with $k_0\geq 2$, then $\varphi$ is not increasing on the whole interval $[0,1]$.
	\end{proposition}

\begin{proof}
We have $(-1)^k\varphi^{(k)}\leq 0$ on $[0,1]$. Suppose for some $\ell \in [k_0, k) \cap \mathbb{N}$, we know that $(-1)^{\ell+1} \varphi^{(\ell+1)}$ is nonpositive on $[0,1]$. Then $(-1)^{\ell} \varphi^{(\ell)}$ is an increasing function on $[0,1]$. Since $(-1)^{\ell} \varphi^{(\ell)}(1) \leq 0$, it follows that $(-1)^{\ell} \varphi^{(\ell)}\leq 0$ on $[0,1]$. Inductively, we see that $(-1)^{k_0} \varphi^{(k_0)}\leq 0$ on $[0,1]$, and so $(-1)^{k_0-1} \varphi^{(k_0-1)}$ is an increasing function (which is also true if $k = k_0$). When $k_0=1$, we obtain the claim.

For the rest of the proof, we only consider $k_0\geq 2$, in which case $\varphi$ is strictly half-monotone. Since $(-1)^{k_0-1} \varphi^{(k_0-1)}(1) > 0$, there must be some $\lambda_{k_0-1} \in [0,1)$ such that $(-1)^{k_0-1} \varphi^{(k_0-1)}$ is nonpositive on $[0, \lambda_{k_0 -1})$ and positive on $(\lambda_{k_0 -1},1]$, where the former interval may be empty.

		Suppose that for some $\ell \in [1, k_0-1) \cap \mathbb{N}$, there exists some $\lambda_{\ell+1} \in [0,1)$ such that $(-1)^{\ell+1} \varphi^{(\ell + 1)}$ is nonpositive on $[0, \lambda_{\ell+1})$ and positive on $(\lambda_{\ell+1},1)$, and therefore $(-1)^{\ell} \varphi^{(\ell)}$ is unimodal on $[0,1]$, and strictly decreasing on $(\lambda_{\ell+1},1]$. Since $(-1)^{\ell}\varphi^{(\ell)}(1) \geq 0$, there exists some $\lambda_{\ell} \in [0,\lambda_{\ell + 1}]$ such that $(-1)^{\ell} \varphi^{(\ell)}$ is nonpositive on $[0, \lambda_{\ell})$ and positive on $(\lambda_{\ell},1)$. Thus, by induction there exists some $\lambda_1 \in [0,1)$ such that $\varphi'$ is nonnegative on $[0, \lambda_1)$ and negative  on $(\lambda_1, 1)$, so $\varphi$ is unimodal on $[0,1]$. This completes the proof.
	\end{proof}

    In Corollary~\ref{cor:Convex_external_fields_-2_s_d-4_inside_sphere}
        and Lemma~\ref{lem:General Inside the Sphere}
        we provide sufficient conditions for $1$ to be the global minimizer of $f$ on $[0,1]$ (i.e. inside the sphere, $\|x\| \leq R$).
        In particular, Corollary~\ref{cor:Convex_external_fields_-2_s_d-4_inside_sphere} implies that convexity of $v$,
        together with
        Theorem~\ref{thm:Necessary Condition Riesz}\ref{item:NecCond SP}, is sufficient.

\begin{corollary}\label{cor:Convex_external_fields_-2_s_d-4_inside_sphere}
    Suppose $-2 < s \leq d-4$ and $v$ is $\mathcal{C}^2$ in the extended sense on $[0, R^2]$, 
    where $R > 0$ satisfies
    Theorem~\ref{thm:Necessary Condition Riesz}\ref{item:NecCond SP}. If $v''$ is nonnegative on $[0, R^2]$, then the modified potential $f$
    defined in \eqref{eq:Modified_potential_f_expression}
    achieves its minimum on $[0,1]$ at $1$.
    
\end{corollary}

\begin{proof}
	From  \eqref{eq:Derivatives of f_s,d,R,v in sphere}, \eqref{eq:positivity of 2F1 on [0,1)}, and \eqref{eq:HG1pos}, we have $f'' \geq 0$ on $[0,1]$. Since $f'(1)=0$, 
 the global minimum of the convex function $f$ on
 the convex set $[0, 1]$ is attained at $1$.
 (See also Proposition \ref{prop:half-monotone implies unimodal}
 which implies $f$ is decreasing on $[0, 1]$.)
\end{proof}

\begin{lemma}\label{lem:General Inside the Sphere}
Suppose $d-4 < s < d-3$, 
and for some $k > 2$, $v$ is $\mathcal{C}^k$ in the extended sense on $[0, R^2]$, where $R > 0$ satisfies
Theorem~\ref{thm:Necessary Condition Riesz}\ref{item:NecCond SP}.
If in addition,
Theorem~\ref{thm:Necessary Condition Riesz}\ref{item:NecCond DD} holds,
$v^{(k)}(R^2 \lambda) \leq 0$ for $\lambda \in [0,1]$, and,
for $3 \leq \ell \leq k-1$,
$f^{(\ell)}(0) \leq 0$,
that is
\begin{equation}\label{eq:d-4_s_d-3_inside_sphere_negative_derivatives_at_0}
    R^{s +2 \ell}v^{(\ell)}(0) \leq -  h_{s,d}^{(\ell)}(0),
\end{equation}
then the modified potential $f$
    defined in \eqref{eq:Modified_potential_f_expression} achieves its minimum on $[0,1]$ at $1$.
\end{lemma}

\begin{proof}
Since $d-4 < s < d-3$,  we see that $\lim_{\lambda \rightarrow 1^-} f^{(k)}(\lambda) = - \infty$  due to \eqref{eq:Derivatives of f_s,d,R,v in sphere}, \eqref{eq:HG1neg}, and the fact that $v^{(k)}$ is nonpositive at $1$. Combining this with \eqref{eq:Derivatives of f_s,d,R,v in sphere}, \eqref{eq:positivity of 2F1 on [0,1)}, and our assumption that $v^{(k)}$ is nonpositive on $[0,1)$, we have that $f^{(k)} \leq 0$ on $[0,1]$.
Let $\varphi(\lambda):=f''(1-\lambda)$. Then, $\varphi$ is half-monotone of order $(1,k-2)$ at $1$. Proposition \ref{prop:half-monotone implies unimodal} implies that $\varphi$ is increasing, or $f''$ is decreasing on $[0,1]$.
By Theorem~\ref{thm:Necessary Condition Riesz}\ref{item:NecCond DD}, $f''(1) \geq 0$, so $f''$ is nonnegative on $[0,1]$.
Consequently, since $f'(1) = 0$ from 
Theorem~\ref{thm:Necessary Condition Riesz}\ref{item:NecCond SP},
it follows that $f$ is decreasing on $[0,1]$, so the minimum of $f$ on $[0, 1]$ occurs at $1$.
\end{proof}

In Lemma \ref{lem:General Outside the Sphere} and
Corollary~\ref{cor:Convex_external_fields_-2_s_d-4_outside_sphere}, we provide
sufficient conditions for $1$ to be the global minimizer of $f$ on $[1, \infty)$ (i.e. outside the sphere, $\|x\| \geq R$).
\begin{lemma}\label{lem:General Outside the Sphere}
Suppose $R > 0$ satisfies
Theorem~\ref{thm:Necessary Condition Riesz}\ref{item:NecCond SP}.
If one of the following sets of conditions holds, then $f$ achieves its infimum on $[1, \infty)$ at $1$.
\begin{enumerate}[label=\textnormal{(\alph*)}]
\item\label{item:General Outside the Sphere case 3} 
$-2 < s < d-4$, $k \in 2\mathbb{N} \cup [2, \frac{d-s}{2})$, $v$ is $\mathcal{C}^k$ in the extended sense on $[R^2, \infty)$ such that $ v^{(k)}(R^2 \lambda) \geq 0$ for $\lambda \in [1, \infty)$, and for $2 \leq \ell < k$, $f^{(\ell)}(1) \geq 0$, that is
     \begin{equation}\label{eq:-2<s<d-4, outside sphere, even derivatives}
        h_{s,d}^{(\ell)}(1) \geq  - R^{2\ell + s} v^{(\ell)}(R^2).
     \end{equation}
\item\label{item:General Outside the Sphere case 4}
$d-4 < s < d-3$, $v$  is $\mathcal{C}^2$ in the extended sense on $[R^2, \infty)$,
Theorem~\ref{thm:Necessary Condition Riesz}\ref{item:NecCond DD}
is satisfied,
and $\lambda\mapsto\lambda^{\frac{s}{2}+2} v''(R^2 \lambda)$ is increasing for $\lambda > 1$.

\end{enumerate}
\end{lemma}
 (The case $s = d-4$ is covered in
 Corollary~\ref{cor:Convex_external_fields_-2_s_d-4_outside_sphere}).

\begin{proof}
We will handle each case individually.

\vspace{1ex}
\underline{Case} \ref{item:General Outside the Sphere case 3}: 
In this case the conditions on $v$ plus  \eqref{eq:Derivatives_of_f_s_d_R_v_outside_sphere}, \eqref{eq:positivity of 2F1 on [0,1)}, and \eqref{eq:HG1pos}, imply that $f^{(k)}(\lambda) \geq 0$ on $[1,\infty)$.  
Suppose that for some $\ell \in [1, k) \cap \mathbb{N}$, we have that $f^{(\ell+1)}(\lambda) \geq 0$ on $[1, \infty)$.  Then $f^{(\ell)}(\lambda)$ is an increasing function, and since $f^{(\ell)}(1) \geq 0$ (due to assumption \eqref{eq:-2<s<d-4, outside sphere, even derivatives} or
Theorem~\ref{thm:Necessary Condition Riesz}\ref{item:NecCond SP}
if $\ell = 1$), we find that $f^{(\ell)}$ is nonnegative on $[1, \infty)$. By induction, $f$ is an increasing function, establishing our claim.

\vspace{1ex}
\underline{Case} \ref{item:General Outside the Sphere case 4}: We want to show that $f''$ is nonnegative for $\lambda \geq 1$, or equivalently that the following quantity is nonnegative:
  \begin{equation*}
     \lambda^{\frac{s}{2} +2} f''(\lambda) =  R^{-s} \frac{s+2}{4}  \; \HG21 \left( 2 + \frac{s}{2}, \frac{2+s-d}{2}; \frac{d}{2}; \lambda^{-1} \right) + R^4 \lambda^{\frac{s}{2} +2} v''(R^2 \lambda).
  \end{equation*}
From
Theorem~\ref{thm:Necessary Condition Riesz}\ref{item:NecCond DD},
this quantity is nonnegative at $\lambda = 1$.
Since $R^4 \lambda^{\frac{s}{2} +2} v''(R^2 \lambda)$ is an increasing function for $\lambda \geq 1$, it is sufficient to show that $\HG21 \left( 2 + \frac{s}{2}, \frac{2+s-d}{2}; \frac{d}{2}; \lambda^{-1} \right)$ is also an increasing function, or equivalently, that $\HG21 \left( 2 + \frac{s}{2}, \frac{2+s-d}{2}; \frac{d}{2}; \lambda \right)$ is a decreasing function on $[0,1)$. Taking a derivative, we have
\begin{equation*}
    \frac{d}{d \lambda} \; \HG21 \left( 2 + \frac{s}{2}, \frac{2+s-d}{2}; \frac{d}{2}; \lambda \right) = - \frac{(s+4)(d-s-2)}{2d} \; \HG21 \left( 3 + \frac{s}{2}, \frac{4+s-d}{2}; \frac{d+2}{2}; \lambda \right).
\end{equation*}
Since $d-4 < s < d-3$,
$c = \frac{d+2}{2} > b = \frac{4+s-d}{2} > 0$,
and $\frac{(s+4)(d-s-2)}{2d} > 0$,
inequality \eqref{eq:positivity of 2F1 on [0,1)} implies the derivative is negative.
Hence $R^{-s} \frac{s+2}{4}  \; \HG21 \left( 2 + \frac{s}{2}, \frac{2+s-d}{2}; \frac{d}{2}; \lambda^{-1} \right)$ is an increasing function on $(1, \infty)$. Thus $f$ is convex, and therefore increasing function on $(1,\infty)$, so its global minimum occurs at $1$.
\end{proof}

As a corollary of Lemma \ref{lem:General Outside the Sphere}\ref{item:General Outside the Sphere case 3},
we have that convexity of $v$,
together with
Theorem~\ref{thm:Necessary Condition Riesz}\ref{item:NecCond SP},
is sufficient on $[1, \infty)$.
\begin{corollary}
    \label{cor:Convex_external_fields_-2_s_d-4_outside_sphere}
    Suppose $-2 < s \leq d-4$, $R > 0$ and  $v$ is $\mathcal{C}^2$ in the extended sense on $[R^2, \infty)$. If $v''$ is nonnegative on $[R^2, \infty)$ where $R > 0$ satisfies
    Theorem~\ref{thm:Necessary Condition Riesz}\ref{item:NecCond SP},
    then $f$ achieves its minimum on $[1,\infty)$ at $1$.
\end{corollary}

\begin{proof}
The case where $-2 < s < d-4$ is handled in Lemma \ref{lem:General Outside the Sphere} case \ref{item:General Outside the Sphere case 3}, so assume that $s = d-4$. Then, from \eqref{eq:Derivatives_of_f_s_d_R_v_outside_sphere}, we see that
\begin{equation*}
f''(\lambda) =  \frac{R^{4-d} (d-2)}{4} \lambda^{- \frac{d}{2}} \Big( 1 - \frac{1}{\lambda} \Big) + R^4 v''(R^2 \lambda),
\end{equation*}
which is nonnegative on $[1, \infty)$. Since $f_{d-4, d, R, v}'(1) = 0$, we see that $f_{d-4, d, R, v}'$ is nonnegative on $[1, \infty)$, so $f_{d-4, d, R, v}$ is increasing, giving us our claim.
\end{proof}

We note that in case \ref{item:General Outside the Sphere case 4} of Lemma \ref{lem:General Outside the Sphere}, since $\lambda^{\frac{s}{2}+2} v''(R^2 \lambda)$ is an increasing function, and it must be at least $0$ at $1$, this does tell us implicitly that $v$ is convex outside the sphere. However, convexity alone is not sufficient here, which is why the range $d-4 < s < d-3$ is excluded in
Corollary~\ref{cor:Convex_external_fields_-2_s_d-4_outside_sphere}.

\subsection{Proof of Theorem \ref{thm:Sphere Min}}\label{subsec:Proof for power law external field}

For this result, we combine Theorem~\ref{thm:Sufficient Conditions Riesz}, case \ref{item:General Outside the Sphere case 4} of Lemma \ref{lem:General Outside the Sphere}, and Lemma \ref{lem:General Inside the Sphere}. The proof is mainly breaking the ranges of $\alpha$ and $s$ into three cases, and showing that each of these cases satisfies some of the results listed above.

Before the proof of Theorem \ref{thm:Sphere Min}, we first show that $\alpha_{s,d}$ in \eqref{eq:Alpha Bound for Sphere} is continuous in $s$, by showing that $\lim_{s \rightarrow 0} \frac{s c_{s,d}}{2- 2 c_{s,d}} = - \frac{1}{2 b_d} $, where $c_{s,d}$ and $b_d$ are defined in  \eqref{eq:c_s,d definition} and \eqref{eq:b_d definition}, respectively. On the one hand, using
\[
  \frac{\partial c_{s,d}}{\partial s} = 
  c_{s,d} \left[\tfrac{1}{2} \psi_0\left(\tfrac{d-s}{2}\right) +
  \tfrac{1}{2} \psi_0\left(d - 1- \tfrac{s}{2}\right) -
  \psi_0(d - s - 1) \right]
\]
we get 
\[
  \lim_{s \to 0} \frac{2(1 - c_{s,d})}{s \, c_{s,d}} =
  -2 \lim_{s \to 0} \frac{ \frac{\partial c_{s,d}}{\partial s}}
  {c_{s,d} + s \frac{\partial c_{s,d}}{\partial s}} =
  \psi_0(d-1) - \psi_0(\tfrac{d}{2}) .
\]
On the other hand, 
$\psi_0(2 z) = \tfrac{1}{2}\left( \psi_0(z) + \psi_0(z + \tfrac{1}{2}) \right) + \log(2)$, see \cite[Eq. 5.5.8]{NIST:DLMF2023}, giving, with $z = \frac{d-1}{2}$,
\[
  b_d = \tfrac{1}{2} \left( \psi_0(\tfrac{d}{2}) - \psi_0(d-1)\right).
\]

\begin{proof}[Proof of Theorem \ref{thm:Sphere Min}]

We have $V(x) = v(\| x \|^2)$, where $v(r) = \frac{\gamma}{\alpha}r^{\alpha/2}$. If for some $R>0$, $\sigma_R$ is the equilibrium measure of $I_{s,V}$, then $R$ must satisfy \eqref{eq:radius for PL}, due to
Theorem~\ref{thm:Necessary Condition Riesz}\ref{item:NecCond SP}.
Then (\ref{eq:energy for PL}) follows from (\ref{eq:radius for PL}) and Proposition \ref{prop:Riesz Energy of a sphere}. It remains to show that $\sigma_R$ is indeed the equilibrium measure of $I_{s,V}$ with $R$ is given by (\ref{eq:radius for PL}). 

Since $\alpha \geq \alpha_{s,d}$
\begin{equation*}
    R^s \big(v(0) - v(R^2) \big)  = -R^{s + \alpha} \frac{\gamma}{\alpha}  = -\frac{c_{s,d}}{2 \alpha}  \geq \begin{cases}
        \frac{c_{s,d} - 1}{s} & s \neq 0\\
        b_d & s = 0
    \end{cases}
\end{equation*}
meaning that
Theorem~\ref{thm:Necessary Condition Riesz}\ref{item:NecCond BC0}
is satisfied. Likewise, we see that
\begin{equation*}
    R^2 \frac{v''(R^2)}{v'(R^2)}  = \frac{1}{2} (\alpha - 2) \geq - \frac{1}{2} \frac{(s+2) (d-s-4)}{2(d-s-3)}
\end{equation*}
so
Theorem~\ref{thm:Necessary Condition Riesz}\ref{item:NecCond DD}
is also satisfied.

First we set up some useful identities. For all $\ell  \in \mathbb{N}$
\begin{equation}\label{eq:Proof of Power law Min, external field derivatives}
    R^{2 \ell} v^{(\ell)}(R^2 \lambda) = R^{\alpha} \frac{\gamma}{2^{\ell}} \lambda^{\frac{\alpha}{2}-\ell} \prod_{j=1}^{\ell-1} (\alpha - 2j) 
\end{equation}
and
\begin{equation*}
   q(\kappa) = R^{s + \alpha} \gamma \kappa^{-\frac{s+\alpha}{2}} = \frac{c_{s,d}}{2}  \kappa^{-\frac{s+\alpha}{2}}
\end{equation*}
where $q$ is as in \eqref{eq:Def of inverse external field}, so for all $\ell \in \mathbb{N}$
\begin{equation}\label{eq:Proof of Power law Min, inverse external field derivatives}
    q^{(\ell)}(\kappa) = \frac{c_{s,d}}{2} \kappa^{-\frac{s+ \alpha}{2} - \ell} (-\frac{1}{2})^{\ell} \prod_{j=0}^{\ell-1} ( \alpha + s + 2j).
\end{equation}

As discussed at the beginning of Section~\ref{sec:Proofs}, the Frostman conditions \eqref{cond: Frostman} and the uniqueness of the equilibrium measure tell us it is sufficient to show that $f$ achieves its global minimum on $[0, \infty)$ at $1$ to prove that $\sigma_R$ is the equilibrium measure. In order to do this, we consider three cases, and for each one, we show that $v$ satisfies certain results from Section~\ref{subsec:sufficient conditions},
which indeed gives us that $1$ is the global minimum of $f$. 

\vspace{2ex}
\textbf{Case 1}: Suppose first that $-2 < s \leq d-4$ and $\alpha \geq 2$. Then $v$ is $\mathcal{C}^2$ in the extended sense and convex on $[0, \infty)$, so the desired result follows from Theorem~\ref{thm:Sufficient Conditions Riesz}.

\vspace{2ex}
\textbf{Case 2}: For our second case, suppose that $-2 < s < d-4$ and $\alpha_{s,d} \leq \alpha < 2$. Let $k = \lceil \frac{d-s}{2} \rceil$. Note that $k \geq 3$ since $s < d-4$.

From \eqref{eq:Proof of Power law Min, external field derivatives}, we have that $(-1)^k v^{(k)}(R^2 \lambda)$ is nonpositive on $[0,1]$.
Since $k = \lceil \frac{d-s}{2} \rceil$,   $(-1)^k f^{(k)}$ is negative on $[0,1]$, by  \eqref{eq:Derivatives of f_s,d,R,v in sphere},  \eqref{eq:positivity of 2F1 on [0,1)}, and \eqref{eq:HG1pos}.
Combining \eqref{eq:Derivatives of f_s,d,R,v in sphere}, \eqref{eq:radius for PL}, \eqref{eq:Proof of Power law Min, external field derivatives}, and \eqref{eq:Equating Higher Geom Series}, we have that for $2 \leq \ell < k$
\begin{align*}
    (-1)^{\ell} R^{s} f^{(\ell)} (1) & = \frac{c_{s,d}}{2^{\ell+1}}  \Big( \prod_{j=1}^{\ell-1} \frac{(d-s-2-2j)(s+2j)}{2 (d-s-2-j)}  - \prod_{j=1}^{\ell-1} (2 j - \alpha) \Big) \\
    & = \frac{c_{s,d}}{2^{\ell+1}}  \Big( \prod_{j=1}^{\ell-1} \frac{(d-s-2-2j)(s+2j)}{2 (d-s-2-j)(2j- \alpha)}  - 1 \Big) \prod_{j=1}^{\ell-1} (2 j - \alpha) .
\end{align*}

We then see that for $1 < j < \frac{d-s-2}{2}$,
\begin{align*}
    \frac{d}{dj} \log\Big( \frac{(d-s-2-2j)(s+2j)}{2 (d-s-2-j)(2j- \alpha)} \Big) & = \frac{-2}{d-s-2-2j} + \frac{2}{d-s-2-j} + \frac{2}{2j + s} - \frac{2}{2j - \alpha} \\
    & = \frac{-2j}{(d-s-2-2j)(d-s-2-j)} - \frac{2(s+ \alpha)}{(2j+s)(2j-\alpha)}
\end{align*}
which is negative, since $2 > \alpha > -s$. Thus, $\frac{(d-s-2-2j)(s+2j)}{2 (d-s-2-j)(2j- \alpha)}$ is strictly decreasing for $j \in [1, k-1]$. 

Now we will show that  $\alpha > 2 - \frac{(s+2)(d-s-4)}{2 (d-s-3)}$. Suppose to the contrary that $\alpha =  2 - \frac{(s+2)(d-s-4)}{2 (d-s-3)}$, so that $(-1)^2 f^{(2)} (1) = 0$. Then, since $\frac{(d-s-2-2j)(s+2j)}{2 (d-s-2-j)(2j- \alpha)}$ is strictly decreasing, we see that  $(-1)^{\ell} f^{(\ell)} (1) < 0$ for $2 < \ell < k$. Recall that $(-1)^{k} f^{(k)} (\lambda) < 0$ on $[0,1]$.
Now, suppose that for some $\ell \in \{2, ..., k-1\}$, we know that$(-1)^{\ell+1} f^{(\ell+1)} (\lambda)$ is negative $[0, 1]$. Then $(-1)^{\ell} f^{(\ell)} (\lambda)$ is strictly increasing on $[0,1]$, and since $(-1)^{\ell} f^{(\ell)} (1) \leq 0$ (the inequality being strict for $\ell > 2$), $(-1)^{\ell} f^{(\ell)} (\lambda)$ is negative on $[0,1)$. Thus, we can conclude inductively that $f^{(2)} (\lambda)$ is negative on $[0,1)$, and so, $f'(\lambda)$ is strictly decreasing on $[0,1]$. Since $f'(1) = 0$, we see that $f$ is strictly decreasing on $[0,1]$, so $f(0) < f(1)$, which is a contradiction to
Theorem~\ref{thm:Necessary Condition Riesz}\ref{item:NecCond BC0}.
Thus, we have that $\alpha > 2 - \frac{(s+2)(d-s-4)}{2 (d-s-3)}$,
and so $(-1)^2 f^{(2)} (1) > 0$.

Since $\frac{(d-s-2-2j)(s+2j)}{2 (d-s-2-j)(2j- \alpha)}$ is strictly decreasing for $j \in [1, k-1]$, and $(-1)^2 f^{(2)} (1) > 0$, we can now conclude that there must be some $\ell_0 \in \{ 2, ..., k-1\}$ such that for $3 \leq \ell \leq \ell_0$ 
\begin{equation*}
    \prod_{j=1}^{\ell-2} \frac{(d-s-2-2j)(s+2j)}{2 (d-s-2-j)(2j- \alpha)}   \leq \prod_{j=1}^{\ell-1} \frac{(d-s-2-2j)(s+2j)}{2 (d-s-2-j)(2j- \alpha)} 
\end{equation*}
and for $\ell_0 < \ell \leq k-1$
\begin{equation*}
    \prod_{j=1}^{\ell-2} \frac{(d-s-2-2j)(s+2j)}{2 (d-s-2-j)(2j- \alpha)}   \geq \prod_{j=1}^{\ell-1} \frac{(d-s-2-2j)(s+2j)}{2 (d-s-2-j)(2j- \alpha)}. 
\end{equation*}
This then implies that there is some $k_0 \in \{3, ..., k\}$ such that $(-1)^{\ell} f^{(\ell)} (1) \leq 0$ for $\ell \in [k_0, k) \cap \mathbb{N}$ and $ (-1)^{\ell} f^{(\ell)} (1) > 0$ for all $\ell \in [2, k_0) \cap \mathbb{N}$.
In other words, $f$ is strictly half-monotone of order $(k_0,k)$ at 1 (note that $f'(1)=0$ due to
Theorem~\ref{thm:Necessary Condition Riesz}\ref{item:NecCond SP}).
By Proposition \ref{prop:half-monotone implies unimodal}, the global minimum of $f$ is at $0$ or $1$; however, by
Theorem~\ref{thm:Necessary Condition Riesz}\ref{item:NecCond BC0},
we know $f(0) \geq f(1)$, giving our desired result for $[0, 1]$.

\bigskip 

Next, consider the function $g$ as in \eqref{eq:def of g_s,d,q}. From \eqref{eq:Proof of Power law Min, inverse external field derivatives}, we have that $(-1)^k q^{(k)}(\kappa)$ is nonnegative on $[0,1]$.
Since $k = \lceil \frac{d-s}{2} \rceil$,
\eqref{eq:Derivatives of g_s,d,q}, \eqref{eq:positivity of 2F1 on [0,1)}, and \eqref{eq:HG1pos}-\eqref{eq:HG1neg} imply that
$(-1)^k g^{(k)}$ is nonnegative on $[0,1]$.
To be precise, we use \eqref{eq:HG1pos} for when $s<d-5$ and $s\neq d-6$; \eqref{eq:HG1zer} for when $s=d-5$ or $s=d-6$; \eqref{eq:HG1neg} for when $d-5<s<d-4$.
Combining \eqref{eq:Derivatives of g_s,d,q}, \eqref{eq:Proof of Power law Min, inverse external field derivatives},  \eqref{eq:Equating Higher Geom Series, ver 2}, and the fact that
\begin{equation*}
    \HG21 \Big( \frac{s}{2} + 1, \frac{2 + s -d}{2}; \frac{d}{2}; 1 \Big) = \frac{c_{s,d}}{2},
\end{equation*}
we have that for $1 \leq \ell < k$
\begin{align*}
    (-1)^{\ell} g^{(\ell)} (1) & = \frac{c_{s,d}}{2^{\ell+1}}  \Big( - \prod_{j=0}^{\ell-1} \frac{(d-s-2-2j)(s+2j+2)}{2 (d-s-3-j)}  + \prod_{j=0}^{\ell-1} ( \alpha + s + 2j) \Big).
\end{align*}

Since $\alpha > 2 - \frac{(s+2)(d-s-4)}{2(d-s-3)}$, we see that
\begin{equation*}
    \alpha + s >  \frac{(d-s-2)(s+2)}{2 (d-s-3)} 
\end{equation*}
and for $j \geq 1$
\begin{equation*}
    \frac{(d-s-2-2j)(s+2j+2)}{2 (d-s-3-j)}  \leq \frac{s}{2} + j + 1 < \alpha + s + 2j.
\end{equation*}
Thus $(-1)^{\ell} g^{(\ell)} (1) \geq 0$ for all $\ell \in \{ 1, ..., k-1\}$. This shows that $-g$ is half-monotone of order $(1,k)$ at $1$. Thus, $g$ is decreasing on $[0,1]$ by Proposition \ref{prop:half-monotone implies unimodal}. From \eqref{eq:equivalence with f and q} and
Theorem~\ref{thm:Necessary Condition Riesz}\ref{item:NecCond SP}, we have $g(1)=2R^sf'(1)=0$, which implies $g\geq 0$ on $[0,1]$. Again by \eqref{eq:equivalence with f and q}, $f$ must be increasing on $[1,\infty)$, giving our desired result.

\vspace{2ex}
\textbf{Case 3}:  For our final case, suppose that $d-4< s < d-3$ and $\alpha_{s,d} \leq \alpha $. Combining \eqref{eq:Derivatives of f_s,d,R,v in sphere}, \eqref{eq:radius for PL}, \eqref{eq:Equating Higher Geom Series}, and the fact that $\alpha \geq \alpha_{s,d} \geq  2 - \frac{(d-s-4)(s+2)}{2 (d-s-3)}$, we have
\begin{equation*}
     R^{s} f'' (1) = \frac{c_{s,d}}{2^{3}}  \Big( \frac{(d-s-4)(s+2)}{2 (d-s-3)}  +(\alpha - 2) \Big) \geq 0.
\end{equation*}
We then see that
\begin{equation*}
    \lambda^{\frac{s}{2}+2}  \frac{d^2}{d \lambda^2} v(R^2 \lambda) =  R^{\alpha} \frac{\gamma(\alpha -2)}{4} \lambda^{\frac{s+\alpha}{2}}
\end{equation*}
which is an increasing function on $[1, \infty)$, since $\alpha > 2 > -s$. Thus, the conditions of Lemma~\ref{lem:General Outside the Sphere} case \ref{item:General Outside the Sphere case 4} have now been satisfied.

Let $k = \lceil \frac{\alpha}{2} \rceil +1$. For $0 \leq \lambda \leq 1$, and $\ell \in \{1, ..., k\}$
\begin{align*}
    \frac{d^{\ell}}{d \lambda^{\ell}} v(R^2 \lambda) = R^{\alpha} \frac{\gamma}{2^{\ell}} \lambda^{\frac{\alpha}{2}-\ell} \prod_{j=1}^{\ell-1} (\alpha - 2j),
\end{align*}
so $v^{(k)}(R^2 \lambda) \leq 0$ on $[0, 1)$, and $v^{(\ell)}(R^2) < \infty$. For $3 \leq \ell \leq k-2$, $v^{(\ell)}(0) = 0 < - h_{s,d}^{(\ell)}(0)$.

We have that
\begin{align*}
    R^s f''(1) & = \frac{c_{s,d}}{4} \left( \frac{(d-s-4)(s+2)}{2(d-s-3)} + \alpha - 2 \right) \\
    & \geq  \frac{c_{s,d}}{4} \left(\frac{(d-s-4)(s+2)}{2(d-s-3)} +  \Big( 2- \frac{(d-s-4)(s+2)}{2(d-s-3)}\Big)   - 2 \right) \\
    & = 0.
\end{align*}
Thus all the conditions of Lemma~\ref{lem:General Inside the Sphere} are satisfied, giving our desired result.
\end{proof}


\subsection{Proof of Proposition \ref{prop:connect to opt control}}\label{subsec:Proof for Wasserstein Result}

\begin{proof}
Note that for $V(x) = \frac{\gamma}{\alpha} \|x\|^{\alpha}$, minimizing $I_{s, V}$ over $\mathcal{P}(\mathbb{R}^d)$ is the same as minimizing the energy over $\mathcal{P}_{\alpha, c}(\mathbb{R}^d)$ for each $c$, then minimizing over positive $c$. So, for $s < d$, $s \neq 0$,
\begin{align*}
\inf_{\mu \in \mathcal{P}(\mathbb{R}^d)} I_{s,V}(\mu) & = 
\inf_{c > 0} \inf_{\mu \in \mathcal{P}_{\alpha, c}(\mathbb{R}^d)} I_{s, 0} (\mu) + 2\frac{\gamma}{\alpha} c^{\alpha} \\
& = \inf_{c > 0} \inf_{\mu \in \mathcal{P}_{\alpha, c}(\mathbb{R}^d)} c^{-s} I_{s} ((c^{-1} Id)_{\#}\mu) + 2\frac{\gamma}{\alpha}c^{\alpha} \\
& = \inf_{c > 0} \Bigg( c^{-s} \inf_{\nu \in \mathcal{P}_{\alpha, 1}(\mathbb{R}^d)} I_{s} (\nu) + 2 \frac{\gamma}{\alpha}c^{\alpha} \Bigg).
\end{align*}
Let $A =\inf\limits_{\nu \in \mathcal{P}_{\alpha, 1}(\mathbb{R}^d)} I_{s} (\nu)$, and assume $A$ is finite. Taking a derivative with respect to $c$, we see that
$$ -s c^{-s-1} A + 2\gamma c^{\alpha -1} = 0$$
is equivalent to
\[
  c =\left( \frac{s A}{2\gamma} \right)^{\frac{1}{s + \alpha}},
\]
so there is exactly one critical point ($s$ and $A$ are either both nonnegative or both nonpositive, so this value is positive). Since $\alpha > -s$, this also means that $c^{-s} A + 2\frac{\gamma}{\alpha} c^{\alpha}$ is strictly decreasing on $\Big( 0, \big( \frac{s A}{2\gamma} \big)^{ \frac{1}{s+\alpha}} \Big)$ and strictly increasing on $\Big( \big( \frac{s A}{2\gamma} \big)^{ \frac{1}{s+\alpha}}, \infty)$, so the minimum occurs at this point. This proves our first claim for $s \neq 0$, and the second claim is a similar proof in reverse.

For the logarithmic case, we have 
\begin{align*}
\inf_{\mu \in \mathcal{P}(\mathbb{R}^d)} I_{s,V}(\mu) & =\inf_{c > 0} \inf_{\mu \in \mathcal{P}_{\alpha, c}(\mathbb{R}^d)} I_{0, 0} (\mu) + 2\frac{\gamma}{\alpha} c^{\alpha} \\
& = \inf_{c > 0} \inf_{\mu \in \mathcal{P}_{\alpha, c}(\mathbb{R}^d)} I_{0} ((c^{-1} Id)_{\#}\mu) - \log(c) + 2\frac{\gamma}{\alpha}c^{\alpha} \\
& = \inf_{c > 0} \inf_{\nu \in \mathcal{P}_{\alpha, 1}(\mathbb{R}^d)} I_{0} (\nu)  - \log(c) + 2 \frac{\gamma}{\alpha}c^{\alpha}.
\end{align*}
Let $A =\inf\limits_{\nu \in \mathcal{P}_{\alpha, 1}(\mathbb{R}^d)} I_{0} (\nu)$, and assume $A$ is finite. Taking a derivative with respect to $c$, we see that
$$ -c^{-1} + 2\gamma c^{\alpha -1} = 0$$
is equivalent to
\[
  c =\left( \frac{1}{2\gamma} \right)^{ \frac{1}{\alpha}},
\]
so there is exactly one critical point. Since $\alpha > 0$, this also means that $ A - \log(c) + 2\frac{\gamma}{\alpha} c^{\alpha}$ is strictly decreasing on $\Big( 0, \big( \frac{1}{2\gamma} \big)^{ \frac{1}{\alpha}} \Big)$ and strictly increasing  on $\Big( \big( \frac{1}{2\gamma} \big)^{ \frac{1}{\alpha}}, \infty \Big)$, so the minimum occurs at this point. This proves our first claim about logarithmic energy, and the second claim is a similar proof in reverse.
\end{proof}

\section{Acknowledgements}

Ryan W. Matzke was supported by the NSF Postdoctoral Fellowship Grant 2202877. The authors would like to thank Rupert Frank, Johannes Hertrich, Arno Kuijlaars, John McCarthy, and Gabriele Steidl for their helpful conversations.
This research includes computations using the computational cluster Katana supported by Research Technology Services at UNSW Sydney.

\appendix

\section{Lennard\,--\,Jones type external fields}
\label{Sec:LJEF}


This section considers some additional examples of Lennard\,--\,Jones type external fields, namely 
with $\alpha > \beta$, and $\gamma, \eta > 0$,
\begin{equation} \label{Eq:EFLJ type}
    v(\rho) = \frac{\gamma}{\alpha} \rho^{\frac{\alpha}{2}} -
            \frac{\gamma \eta}{\beta} \rho^{\frac{\beta}{2}},
            \qquad \rho = \|x\|^2.
\end{equation}
We present a case where the equilibrium measure has compact support despite the fact that $v(\infty) := \lim_{\rho\to\infty} v(\rho) = 0$. This can only occur if $0 > \alpha$, in which case Theorem~\ref{thm:Sufficient Conditions Riesz}, which requires $v$ to be convex, fails.
We also provide examples when there are no solutions or multiple solutions for equation \eqref{Eq:EFLJ R eqn}, which must be satisfied by an optimal sphere radius $R_*$.

Note that
\begin{equation}\label{eq:EFLJ deriv}
     v'(\rho) = \frac{\gamma}{2} \rho^{\frac{\beta}{2}-1}
     \left( \rho^{\frac{\alpha-\beta}{2}} - \eta \right),
     \qquad
     v''(\rho) = \frac{\gamma}{4} \rho^{\frac{\beta}{2}-2}
     \left( (\alpha-2)\rho^{\frac{\alpha-\beta}{2}} - \eta (\beta-2) \right),
\end{equation}
so $v$ has a minimum 
$v^* = -\gamma \left(\frac{\alpha-\beta}{\alpha\beta}\right) \eta^\frac{\alpha}{\alpha-\beta} < 0$
at $\rho^* = \eta^\frac{2}{\alpha-\beta}$.
Moreover, for $\alpha \geq 2\geq \beta$, $v''(\rho)>0$ for
all $\rho > 0$, so $v$ is convex on $(0, \infty)$,
while for $\alpha < 2$, there is a single point of inflection
\[
  \tilde\rho := \left( \eta \; \frac{2-\beta}{2-\alpha} \right)^\frac{2}{\alpha-\beta} ,
\]
so $v$ is convex on $[0, \tilde\rho]$ and concave on $[\tilde\rho, \infty)$.



\subsection{A particular case with sphere radius $1$}

Let $- 2 < s \leq d-4$, $\beta = -b-s$ for some $b > 2$ and $\alpha = -2 -s$, and $\gamma, \eta > 0$.
The necessary condition
Theorem~\ref{thm:Necessary Condition Riesz}\ref{item:NecCond SP}
is that $R_*$ is a solution of
\begin{equation}\label{eq:LJ Reqn}
 R^{-2} - \eta R^{-b} = R^{\alpha+s} - \eta R^{\beta +s} = \frac{c_{s,d}}{2 \gamma}.
 \end{equation}
If we choose $\gamma, \eta > 0$ such that
\begin{equation}\label{eq:R1etagam}
 1 - \eta = \frac{c_{s,d}}{2\gamma},
\end{equation}
then one solution is $R_* = 1$, for which $\lambda = \frac{\| x \|^2}{R_*^2} = \rho$, see \eqref{Eq:EFLJ type}.

\begin{corollary}[A Special Lennard--Jones Field]\label{Cor:LJ suff neg alpha}
Let $-2 < s \leq d-4$,
\begin{equation}\label{Cor A1 gamma lower bound}
   \gamma > \frac{c_{s,d}}{2}
   \max \left\{1,
   \frac{(2b + s)(2 + s)}{s (b-2)} \right\},
\end{equation}
$\eta = 1 - \frac{c_{s,d}}{2\gamma}$,
$\alpha = -2 - s$ and $\beta = -b -s$ with
\[
  b > \max \left\{2, \ 
  \left(\frac{s+4}{\eta} - s -2\right),\ 
  \frac{1}{\eta}  \left( \frac{(d-s-2) (s+2)}{d \gamma} \; \HG21 \Big( \frac{s+4}{2} , \frac{4 + s -d}{2}; \frac{d+2}{2}; 1 \Big) + 2 \right) \right\}.
\]
Then
$\sigma_1$ is the equilibrium measure for the Lennard--Jones type external field \eqref{Eq:EFLJ type}. 
\end{corollary}

\begin{proof}
If $b \geq \frac{s+4}{\eta} - s -2$, then the points of inflection $\tilde\lambda = \tilde\rho \geq 1$, so $v$ is convex on $[0, 1]$
and, by Corollary~\ref{cor:Convex_external_fields_-2_s_d-4_inside_sphere},
the modified potential $f$ attains its global minimum on $[0, 1]$ at $1$.

With $q$ as in \eqref{eq:Def of inverse external field},
\eqref{eq:EFLJ deriv} and $\beta - \alpha = 2 - b$,
\[
  q(\kappa) = 2  \kappa^{- \frac{s}{2}-1}v'( \kappa^{-1})
  = \gamma \kappa \left(1 - \eta \kappa^{\frac{b}{2}-1}\right),
\]
so
\[
  q'(\kappa) = \frac{\gamma}{2 } \left(2 - \eta b \kappa^{\frac{b-2}{2}}\right),
  \qquad
  q''(\kappa) = - \frac{\gamma \eta}{4} b (b-2) \kappa^{\frac{b-4}{2}},
\]
and $b > 2$ implies that $q''$ is negative on $(0,1]$. With $g$ as in \eqref{eq:def of g_s,d,q}, we see that $g'' < 0$ on $(0,1]$ as well, due to \eqref{eq:Derivatives of g_s,d,q} and \eqref{eq:positivity of 2F1 on [0,1)}.

If
\[
  b > \frac{1}{\eta}  \left( \frac{(d-s-2) (s+2)}{d \gamma} \;
       \HG21 \Big( \frac{s+4}{2} , \frac{4 + s -d}{2}; \frac{d+2}{2}; 1 \Big) + 2 \right),
\]
then, using \eqref{eq:Derivatives of g_s,d,q},
\begin{equation*}
-g'(1) = - \frac{1}{2} \left(  \frac{(d-s-2) (s+2)}{d} \right) \; \HG21 \Big( \frac{s+4}{2} , \frac{4 + s -d}{2}; \frac{d+2}{2}; 1 \Big) + \frac{\gamma}{2 } ( \eta b - 2) > 0.
\end{equation*}
Thus, by Proposition \ref{prop:half-monotone implies unimodal}, $g$ is unimodal, and since $g(1) = 0$, due to \eqref{eq:equivalence with f and q} and the fact that $f'(1) = 0$, we know that there is some $\kappa \in [0,1]$ so that $g$ is nonpositive on $[0, \kappa)$ and nonnegative on $[\kappa, 1]$. Again using \eqref{eq:equivalence with f and q}, we see that $f'$ is nonnegative on $[1, \frac{1}{\kappa})$ and nonpositive on $( \frac{1}{\kappa}, \infty)$, with the second interval being empty if $\kappa = 0$. Thus, $f$ is unimodal on $[1, \infty)$, 
and, with $R_* = 1$,
\[
  f(1) = \frac{c_{s,d}}{s} + v(1) =
  \frac{c_{s,d}}{2}\left(\frac{2b+s}{s(b+s)} \right)
  - \gamma \frac{(b-2)}{(b+s)(2+s)}
  <  0 = \lim_{\lambda \to \infty} f(\lambda),
\]
so the infimum of $f$ on $[1, \infty)$ is attained at $1$.
This then implies that $\mueq = \sigma_1$.
\end{proof}

\subsection{A numerical example}
\label{Sec:LJEF num ex}

Let $d = 8$, $s = 4$, so $c_{s,d} = \frac{1}{2}$ and consider the
Lennard\,--\,Jones type external field \eqref{Eq:EFLJ type} with parameters
\[
  \alpha = -6, \quad \beta = -12, \quad \gamma = 5 \quad\mbox{and}\quad \eta = \frac{19}{20}.
\]
These values of $\alpha, \beta$ correspond to the classical Lennard\,--\,Jones field, see \cite{FischerWendland2023potentials} for example.
Then \eqref{eq:R1etagam} is satisfied, so $R_* = 1$ satisfies \eqref{eq:LJ Reqn},
and $\lambda = 1$ is the global minimum of the corresponding modified potential, as illustrated in Figure~\ref{fig:LJ example}~(B).
These parameter values give $b = 8$,
the bound \eqref{Cor A1 gamma lower bound}
is $\gamma > \frac{5}{4}$ and all the conditions of Corollary~\ref{Cor:LJ suff neg alpha} are satisfied, so $\sigma_1$ is the equilibrium measure.
Note, however, that there is another solution to \eqref{eq:LJ Reqn},
around $R = 4.47$, as illustrated in Figure~\ref{fig:LJ example}~(A).
\begin{figure}[ht]
\centering
\begin{subfigure}[t]{0.47\textwidth}
\centering
    \includegraphics[width=\textwidth]{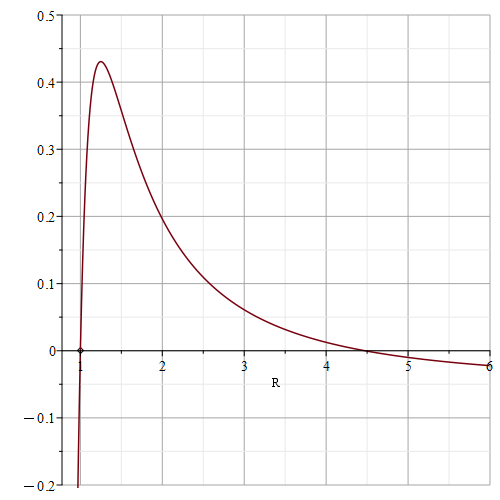}
    \caption{\label{fig:LJ example_R_eqn}
    Graph of $R^{s+2} v''(R) - \frac{c_{s,d}}{4}$ whose zeros satisfy Theorem~\ref{thm:Necessary Condition Riesz}\ref{item:NecCond SP}   for the optimal radius, including $R_* = 1$}
\end{subfigure}
\quad
\begin{subfigure}[t]{0.47\textwidth}
\centering
\includegraphics[width=\textwidth]{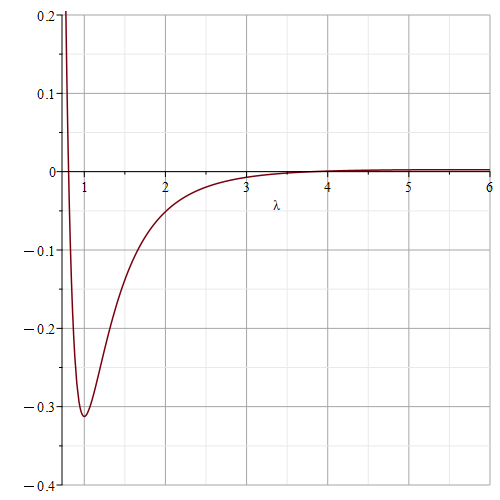}
\caption{\label{fig:LJ example_mod_pot}
The modified potential $f(\lambda)$ for $R_* = 1$ defined in
\eqref{eq:Modified_potential_f_expression},
with global minimum at $\lambda = 1$}
\end{subfigure}
\caption{\label{fig:LJ example}
 Lennard--Jones type external field with $d = 8$, $s = 4$, $\alpha = -6$, $\beta = -12$,
$\gamma = 5$ and $\eta = \frac{19}{20}$}
\end{figure}

For these parameter values,
but with  $\gamma = 1$ and $\eta = \frac{3}{4}$
then \eqref{eq:R1etagam} is satisfied so $R_* = 1$ is a solution of \eqref{eq:LJ Reqn},
but $\lambda = 1$ is a local, not global, minimizer of the modified potential,
as \eqref{Cor A1 gamma lower bound} and hence Theorem~\ref{thm:Necessary Condition Riesz}\ref{item:NecCond BCinf}
are not satisfied.
Also, for $\gamma = \frac{1}{3}$ and $\eta = \frac{1}{2}$, equation \eqref{eq:LJ Reqn} has no solutions.
In these instances the support of the equilibrium measure is not a sphere.

\section{Properties of Hypergeometric Functions}
\label{se:appendix HG}

\subsection{Gauss hypergeometric function}
Throughout this paper, we make use of two types of hypergeometric series,
\begin{equation}\label{eq:HG}
  \HG21(a, b; c; z) := \sum_{k=0}^\infty \frac{(a)_k (b)_k}{(c)_k} \frac{z^k}{k!},
\end{equation}
and
\begin{equation}\label{eq:HG32}
  \HG32(a, b, p; c,q; z) := \sum_{k=0}^\infty \frac{(a)_k (b)_k (p)_k}{(c)_k (q)_k} \frac{z^k}{k!},
\end{equation}
where $(a)_k = a (a+1) \cdots (a+k-1)$ is the Pochhammer symbol, 
which are both absolutely convergent for $|z| < 1$.

\subsection{Behavior at $1$}

The behaviour at $z = 1$ depends on the value of $c - a -b$:
\begin{itemize}
\item If $\Re(c - a - b) > 0$, then \cite[15.4.20]{NIST:DLMF2023} (also known
  as Gauss summation theorem)
\begin{equation}\label{eq:HG1pos}
  \HG21(a, b; c; 1) = \frac{\Gamma(c) \Gamma(c-a-b)}{\Gamma(c-a) \Gamma(c-b)}.
\end{equation}
\item
If $c = a + b$, then \cite[15.4.21]{NIST:DLMF2023} 
\begin{equation}\label{eq:HG1zer} 
  \lim_{z\to 1^-} \frac{\HG21(a, b; a+b; z)}{-\log(1 - z)} =
  \frac{\Gamma(a+b)}{\Gamma(a) \Gamma(b)}.
\end{equation}
\item
If $\Re(c - a - b) < 0$, then \cite[15.4.23]{NIST:DLMF2023} 
\begin{equation}\label{eq:HG1neg}
  \lim_{z\to 1^-} \frac{\HG21(a, b; c; z)}{(1 - z)^{c-a-b}} =
  \frac{\Gamma(c) \Gamma(a+b-c)}{\Gamma(a) \Gamma(b)}.
\end{equation}
\end{itemize}
As an immediate consequence, we have that for all $k \in \mathbb{N}_0$, if $s < d-k-2$,
\begin{align}\label{eq:Equating Higher Geom Series}
\frac{1}{2} \; & \HG21\left(\frac{s}{2}+k, \frac{2+s-d}{2} + k; \frac{d}{2}+k;1\right) \nonumber \\
& = \frac{d-s-k-2}{d+2k} \; \HG21\left(\frac{s}{2}+k+1, \frac{2+s-d}{2} + k+1;\frac{d}{2}+k+1;1\right).
\end{align}
Likewise, for all $k \in \mathbb{N}_0$, if $s < d-k-3$,
\begin{align}\label{eq:Equating Higher Geom Series, ver 2}
\frac{1}{2} \; & \HG21\left(\frac{s}{2}+k+1, \frac{2+s-d}{2} + k; \frac{d}{2}+k; 1\right) \nonumber \\
& = \frac{d-s-k-3}{d+2k} \; \HG21\left(\frac{s}{2}+k+2, \frac{2+s-d}{2} + k+1;\frac{d}{2}+k+1;1\right).
\end{align}

\subsection{Euler integral representation}
Euler's integral formula \cite[15.6.1]{NIST:DLMF2023},
for $\Re(c) > \Re(b) > 0$ and $|\arg(1-z)| < \pi$, is
\begin{equation} \label{eq:EIF}
\HG21(a, b; c; z) = \frac{\Gamma(c)}{\Gamma(b)\Gamma(c-b)}
\int_0^1  \frac{u^{b-1} (1 - u)^{c-b-1}}{(1 - z u)^a} \rd u .
\end{equation}
This formulation leads us to the following result, which is used repeatedly throughout the text: 
\begin{equation}\label{eq:positivity of 2F1 on [0,1)}
    \HG21(a, b; c; z) > 0 \text{ for } c > b > 0 \text{ and } z \in [0,1).
\end{equation}
A $\HG32$ can be written with a similar integral formula, involving a $\HG21$  (see \cite[16.5.2]{NIST:DLMF2023}  with $p=2$, $q=1$): For $\Re (b_0) > \Re (a_0) >0$ and
$|\arg(1-z)|<\pi$,
\begin{equation*}
  \HG32(a_0,a_1,a_2;b_0,b_1;z)
  =\frac{\Gamma(b_0)}{\Gamma(a_0)\Gamma(b_0-a_0)}
  \int_0^1t^{a_0-1}(1-t)^{b_0-a_0-1}
  {}_2\rF_1(a_1,a_2;b_1;zt)\rd t
\end{equation*}
In particular, for $a_1=a_2=1$ and $b_1=2$
and using ${}_2\rF_1(1,1;2;z)=-\frac{\log(1-z)}{z}$, we find 
\begin{equation}\label{eq:EIF32log}
  \HG32(a_0,1,1;b_0,2;z)
  =-\frac{\Gamma(b_0)}{z\Gamma(a_0)\Gamma(b_0-a_0)}
  \int_0^1t^{a_0-2}(1-t)^{b_0-a_0-1}
  \log(1-zt)\rd t.
\end{equation}

\subsection{Derivatives}
The derivative of a hypergeometric function is, \cite[15.5.1]{NIST:DLMF2023},
\begin{equation} \label{eq:HGderiv}
  \frac{\rd}{\rd z} \; \HG21(a, b; c; z) = \frac{ab}{c} \; \HG21(a+1, b+1; c+1; z).
\end{equation}

\begin{equation*}
    \frac{\rd}{\rd z} \; \HG32(a, b, c; p,q; z) = \frac{abc}{pq} \; \HG32(a+1, b+1, c+1; p+1,q+1; z).
\end{equation*}

We also have
\begin{equation}\label{eq:HGderivFor z inverse}
    \frac{\rd}{\rd z} ( z^{-a} \; \HG21( a,b;c; z^{-1}) = -a z^{-a-1} \; \HG21( a+1,b;c; z^{-1})
\end{equation}
which follows from \cite[Eq 15.2.3]{AbrS64}.

\subsection{Gauss quadratic transformation}
See \cite[2.11 (5)]{ErdMOT53} 
\begin{equation} \label{eq:GHGquad}
  \HG21\left(a, b; 2b; \frac{4z}{(1+z)^2}\right) = 
  (1 + z)^{2a} \; \HG21\left(a, a-b+\tfrac{1}{2}; b+\tfrac{1}{2}; z^2\right),
  \quad z\in[0, 1],
\end{equation}
with the series absolutely convergent at $z = 1$ when $b-a > 0$.



\section{Riesz energy and potential}\label{sec:Riesz Energy of Sphere}

\subsection{Legendre duplication formula for the Euler Gamma function}
The Legendre duplication formula, \cite[Eq.~5.5.5]{NIST:DLMF2023}, is
\begin{equation}\label{eq:LegDup}
  \Gamma(\tfrac{1}{2})\Gamma(2 z) = 2^{2z-1}\Gamma(z)\Gamma(z+\tfrac{1}{2}),
  \quad\text{for $2z \neq 0, -1, -2, \ldots$}.
\end{equation}

\subsection{Funk\,--\,Hecke formula}
See, for instance, {\cite[p.~18]{Mul66},
  \cite[Eq.~(5.1.9),~p.~197]{BorHS19}}. Recall that $\sigma_1$ denotes the
uniform probability measure on $\dS^{d-1}$, $d\geq2$. Then, for all
$x\in\dS^{d-1}$, 
\begin{equation}\label{eq:funkhecke}
  \int_{\dS^{d-1}} p(x\cdot y)\sigma_1(\rd y)
  =\tau_{d-1}   \int_{-1}^1 p(t)(1-t^2)^{\frac{d-3}{2}}\rd t,
\end{equation}
where
\begin{equation}\label{eq:Funk-Hecke constant}
 \frac{1}{\tau_{d-1}} = \int_{-1}^{1}(1-t^2)^{\frac{d-3}{2}} \rd t =  \mathrm{Beta}(\tfrac{1}{2},\tfrac{d-1}{2})  :=\frac{\Gamma(\frac{1}{2})\Gamma(\frac{d-1}{2})}{\Gamma(\frac{d}{2})}.   
\end{equation}
In probabilistic terms, this means that if $Y$ is a random vector of
$\mathbb{R}^d$ uniformly distributed on $\dS^{d-1}$ then for all
$x\in\dS^{d-1}$, the law of $x\cdot Y$ has density
$\tau_{d-1}(1-t^2)^{\frac{d-3}{2}}\mathbbm{1}_{t\in[-1,1]}$. This is an
arcsine law when $d=2$, a uniform law when $d=3$, a semicircle law when $d=4$,
and more generally, for arbitrary values of $d\geq2$, a Beta law on $[-1,1]$.

\subsection{Riesz Potential for uniform measure on a Sphere}\label{sec:Appendix Pot of Sphere}


\begin{proposition}\label{prop: Riesz/Log Potential of a Sphere}
For $-2 < s < d-1$ and $R >0$
\begin{equation*}
    U_{s}^{\sigma_R}(x) =
    \begin{cases}
        R^{-s} \hsd{\frac{\|x\|^2}{R^2}} & s \neq 0, \\
        -\log(R) + \hzd{\frac{\|x\|^2}{R^2}} & s = 0.
    \end{cases}
\end{equation*}
\end{proposition}
This follows immediately from Lemmas \ref{lem:Riesz Potential of Sphere} and \ref{lem:Log Potential of Sphere}, and as a consequence, we can easily compute then Riesz energy of $\sigma_R$ using \eqref{eq:Riesz Sphere Potential Function no Radius}.


\begin{proposition}\label{prop:Riesz Energy of a sphere}
Let $-2 < s < d-1$ and $R > 0$. Then
\begin{equation*}
    I_{s}(\sigma_R) = \begin{cases}
        \frac{R^{-s}}{s} c_{s,d}, & s \neq 0\\
        - \log(R) + b_{d}, & s = 0
    \end{cases}.
\end{equation*}
\end{proposition}

\begin{lemma}\label{lem:Riesz Potential of Sphere}
For $s < d-1$, $s\neq 0$, $R > 0$, and $\rho = \frac{\|x\|}{R}$
\begin{align}
\label{eq:FHHGa}
  U_s^{\sigma_R}(x) & =
  \frac{R^{-s}}{s} (1 + \rho)^{-s} \; \HG21\left(\tfrac{s}{2}, \tfrac{d-1}{2}; d-1; \frac{4 \rho}{(1+ \rho)^2}\right),  \\[1ex]
    \label{eq:FHHGb}
   & = \begin{cases}
   \displaystyle
  \frac{R^{-s}}{s} \; \HG21\left(\tfrac{s}{2}, \tfrac{s-d}{2}+1; \tfrac{d}{2}; \rho^2 \right) &
   0 \leq \rho \leq 1,\\[1ex]
   \displaystyle
   \frac{R^{-s}}{s} \rho^{-s} \; \HG21\left(\tfrac{s}{2}, \tfrac{s-d}{2}+1; \tfrac{d}{2}; \rho^{-2} \right) &
   \rho \geq 1,
   \end{cases}.
\end{align}
\end{lemma}

\begin{proof}
For $x =0$, the potential $U_{s}^{\sigma_R}$ is $\frac{R^{-s}}{s}$.

For $x \neq 0$, using the Funk-Hecke formula \eqref{eq:funkhecke} 
and the substitution $t = 2 u - 1$, gives
\begin{align*}
\int_{\mathbb{R}^d} \frac{1}{s} \|x - y\|^{-s} \rd \sigma_R(y) & = \int_{\mathbb{S}^{d-1}} \frac{1}{s} \|x - Ry\|^{-s} \rd \sigma_1(y) \\
& = \int_{\mathbb{S}^{d-1}} \frac{R^{-s}}{s} \Big( \frac{\|x\|^2}{R^2} + \|y\|^2 - 2 \frac{\|x\|}{R} \Big\langle \frac{x}{\|x\|}, y \Big\rangle \Big)^{- \frac{s}{2}} \rd \sigma(y) \\
& = \tau_{d-1}  \frac{R^{-s}}{s} \int_{-1}^{1} ( 1 + \rho^2 -  2 \rho t)^{- \frac{s}{2}} (1-t^2)^{\frac{d-3}{2}} \rd t\\
 & =   2 \tau_{d-1}  \frac{R^{-s}}{s} \int_0^1 \left( (1+\rho)^2 - 4 \rho u\right)^{-\frac{s}{2}} (4 u (1-u))^\frac{d-3}{2} \rd u\\[1ex]
   & = 2^{d-2} \tau_{d-1}  \frac{R^{-s}}{s} (1+\rho)^{-s} \int_0^1 \left(1 - \frac{4\rho u}{(1+\rho)^2}\right)^{-\frac{s}{2}}
       u^\frac{d-3}{2} (1-u)^\frac{d-3}{2} \rd u .
\end{align*}
Euler's integral formula \eqref{eq:EIF} with
$a = \frac{s}{2}$, $b = \frac{d-1}{2}$, and $c = d-1$,
and \eqref{eq:Funk-Hecke constant}
plus the Legendre duplication formula \eqref{eq:LegDup}
with $z = \frac{d-1}{2}$, gives \eqref{eq:FHHGa}.
The quadratic transformation \eqref{eq:GHGquad}
gives the first case in \eqref{eq:FHHGb}
with $a = \frac{s}{2}$, $b = \frac{d-1}{2}$,
so $b-a = \frac{d-1-s}{2} > 0$ if and only if $ s < d-1$.
The second case follows using the transformation $u = 1/\lambda$.
\end{proof}

\begin{lemma}\label{lem:Log Potential of Sphere}
For  $s = 0$, $R > 0$, and $\rho = \frac{\|x\|}{R}$
\begin{align}
\label{eq:FHHGLog}
  \int_{\mathbb{R}^d} -\log( \|x - y\|) \rd \sigma_R(y) 
   =  &
  -\log(R) -  \log(1 + \rho) + \nonumber \\
  & \frac{\rho}{(1+ \rho)^2} \; \HG32\left(1, 1,  \tfrac{d+1}{2}; 2, d; \frac{4 \rho}{(1+ \rho)^2}\right). 
\end{align}
\end{lemma}

\begin{proof}
For $x = 0$, the potential is clearly $-\log(R)$.

For $x \neq 0$, using the Funk-Hecke formula \eqref{eq:funkhecke}, Euler's integral representation \eqref{eq:EIF}, and \eqref{eq:Funk-Hecke constant}, with
the substitution $t = 2 u - 1$, gives
\begin{align*}
-\int_{\mathbb{R}^d} & \log( \|x - y\|) \rd \sigma_R(y)  = -\int_{\mathbb{S}^{d-1}} \log(\|x - Ry\|) \rd \sigma_1(y) \\
& = -\int_{\mathbb{S}^{d-1}} \Bigg( \log(R) + \frac{1}{2} \log\Big( \frac{\|x\|^2}{R^2} + \|y\|^2 - 2 \frac{\|x\|}{R} \Big\langle \frac{x}{\|x\|}, y \Big\rangle \Big) \rd \sigma(y) \\
& = -\tau_{d-1}  \int_{-1}^{1} \Big( \log(R) + \frac{1}{2}\log( 1 + \rho^2 -  2 \rho t) \Big) (1-t^2)^{\frac{d-3}{2}} \rd t\\
 & = - \log(R) -   \tau_{d-1}  \int_0^1 \log\left( (1+\rho)^2 - 4 \rho u\right) (4 u (1-u))^\frac{d-3}{2} \rd u\\[1ex]
   & = - \log(R) - \log(1+\rho) -
    2^{d-3} \tau_{d-1} \int_0^1 \log\left(1 - \frac{4\rho u}{(1+\rho)^2}\right)
       u^\frac{d-3}{2} (1-u)^\frac{d-3}{2} \rd u .
\end{align*}
Euler's integral formula \eqref{eq:EIF32log} with $a_0 = \frac{d+1}{2}$ and $b_0 = d$ plus the Legendre duplication formula \eqref{eq:LegDup}
with $z = \frac{d-1}{2}$ gives \eqref{eq:FHHGLog}.
\end{proof}

Thus, for $s < d-1$, setting
\begin{equation*}
g_{s,d}(\rho) :=  \begin{cases}
 \frac{1}{s} (1 + \rho)^{-s} \; \HG21\left(\tfrac{s}{2}, \tfrac{d-1}{2}; d-1; \frac{4 \rho}{(1+ \rho)^2}\right) & s \neq 0 \\
    -  \log(1 + \rho) + \frac{\rho}{(1+ \rho)^2} \; \HG32\left(1, 1,  \tfrac{d+1}{2}; 2, d; \frac{4 \rho}{(1+ \rho)^2}\right) & s= 0
\end{cases},
\end{equation*}
and taking the derivative (and after series expansion and some algebra)
we get the single formula for $-2 < s < d-2$
\begin{align*}
    g_{s,d}'(\rho) & = -\frac{1}{(1+ \rho)^{s+3}} \Bigg( (1+ \rho)^{2} \; \HG21 \left( \frac{s}{2}, \frac{d-1}{2};d-1; \frac{4 \rho}{(1+ \rho)^2} \right)\nonumber\\
    &\qquad\qquad+ (\rho-1) \; \HG21 \left( 1+\frac{s}{2}, \frac{d+1}{2};d; \frac{4 \rho}{(1+ \rho)^2} \right) \Bigg)\\
    & = \begin{cases}
        - \frac{d-s-2}{d} \; \rho \;
        \HG21 \left( 1+ \frac{s}{2}, \frac{4+s-d}{2}; \frac{d+2}{2}; \rho^2 \right) & \rho \leq 1 \\
        - \rho^{-s-1} \; \HG21 \left( 1+\frac{s}{2}, \frac{2+s-d}{2}; \frac{d}{2}; \rho^{-2} \right)  & \rho \geq 1 
    \end{cases}.
\end{align*}

Let us define $h_{s,d}: [0, \infty) \rightarrow \mathbb{R}$ by, 
\begin{equation*}
\hsd{\lambda} = \begin{cases}
 \frac{1}{s} (1 + \sqrt{\lambda})^{-s} \; \HG21\left(\tfrac{s}{2}, \tfrac{d-1}{2}; d-1; \frac{4 \sqrt{\lambda}}{(1+ \sqrt{\lambda})^2}\right) & s \neq 0 \\
    -  \log(1 + \sqrt{\lambda}) + \frac{\sqrt{\lambda}}{(1+ \sqrt{\lambda})^2} \; \HG32\left(1, 1,  \tfrac{d+1}{2}; 2, d; \frac{4 \sqrt{\lambda}}{(1+ \sqrt{\lambda})^2}\right) & s= 0
\end{cases}.
\end{equation*}
In particular we obtain, using $\hsd{\lambda} = g_{s,d}(\sqrt{\lambda})$, that for $-2 < s < d-2$
\begin{equation*}
     \hsdp{\lambda}  = \begin{cases}
        - \frac{d-s-2}{2d} \; \HG21 \left( 1+ \frac{s}{2}, \frac{4+s-d}{2}; \frac{d+2}{2}; \lambda \right) & \lambda \leq 1 \\
        - \frac{1}{2}\lambda^{-\frac{s}{2}-1} \; \HG21 \left( 1+ \frac{s}{2}, \frac{2+s-d}{2}; \frac{d}{2}; \lambda^{-1} \right) & \lambda \geq 1 
    \end{cases}.
\end{equation*}
This expression is useful in determining if $\sigma_R$ is the equilibrium measure, since $h_{s,d}$ describes the potential of the uniform measure on the unit sphere.

\ifAddSuffCond
\section{Additional sufficient conditions}
\label{Ap:Ex not Suff}

In this appendix, we gather some additional results giving conditions for which $f$ achieves its minimum on $[0,1]$ or $[1, \infty)$ at $1$, which are not used elsewhere in the paper, but may be useful for the study of other external fields.

We first provide some sufficient conditions for the behavior inside the sphere.
\begin{lemma}\label{lem:Extra Inside the Sphere}
If one of the following conditions hold, then $f$ achieves its minimum on $[0,1]$ at $1$.
\begin{enumerate}[label=\textnormal{(\alph*)}]
\item\label{item:General Inside the Sphere case 1} For $-2 < s < d-4$, $R> 0$, $k = \lceil \frac{d-s}{2} \rceil$, and $v$  is $\mathcal{C}^k$ in the extended sense on $[0, R^2]$, with $(-1)^k v^{(k)}$ nonpositive on $[0,R^2]$, such that \ref{item:NecCond SP} and \ref{item:NecCond BC0} in
Theorem~\ref{thm:Necessary Condition Riesz} are satisfied.
In addition, there is some $k_0 \in \{3, ..., k\}$ such that for $k_0 \leq \ell < k$, $(-1)^{\ell} f^{(\ell)}(1) \leq 0$, that is
\begin{equation}\label{eq:Completely_monotonic_external_fields_-2_s_d-4_inside_sphere_eq 1}
   (-1)^{\ell+1} v^{(\ell)}(R^2) \geq (-1)^{\ell} R^{-s -2 \ell} h_{s,d}^{(\ell)}(1),
\end{equation}
and for $2 \leq \ell < k_0$, $(-1)^{\ell} f^{(\ell)}(1) \geq 0$, that is
\begin{equation}\label{eq:Completely monotonic_external_fields_-2_s_d-4_inside_sphere_eq 2}
   (-1)^{\ell+1} v^{(\ell)}(R^2) \leq (-1)^{\ell} R^{-s -2 \ell} h_{s,d}^{(\ell)}(1).
\end{equation}
\item\label{item:General Inside the Sphere case 2}   For $-2 < s \leq d-4$, $R> 0$, and for some $2 \leq k \leq \lfloor \frac{d-s}{2} \rfloor$, $v$  is $\mathcal{C}^k$ in the extended sense  on $[0, R^2]$, with $(-1)^k v^{(k)}$ nonnegative on $[0,R^2]$. In addition,
Theorem~\ref{thm:Necessary Condition Riesz}\ref{item:NecCond SP}
is satisfied, and for $2 \leq \ell \leq k-1$,
$(-1)^{\ell} f^{(\ell)}(1) \geq 0$, that is
\begin{equation}\label{eq:Pos_derivative_-2_s_d-4_inside_sphere} 
   (-1)^{\ell+1} v^{(\ell)}(R^2) \leq (-1)^{\ell} R^{-s -2 \ell} h_{s,d}^{(\ell)}(1)
\end{equation}
\item\label{item:Extra Inside the Sphere case 1} For $d-4 < s < d-3$, $R>0$, $k>2$, and $v$ is $\mathcal{C}^k$ in the extended sense on $[0, R^2]$, with $v^{(k)}$ nonnegative on $[0,R^2]$, such that conditions \ref{item:NecCond SP}, \ref{item:NecCond DD}, and \ref{item:NecCond BC0} of
Theorem~\ref{thm:Necessary Condition Riesz}
are satisfied.
In addition, for $3 \leq \ell \leq k-1$, $v^{(\ell)}(R^2) < \infty$ and $f^{(\ell)}(0) \geq 0$, that is
\begin{equation}\label{eq:d-4<s<d-3,insidesphere}
    R^{s +2 \ell}v^{(\ell)}(0) \geq -  h_{s,d}^{(\ell)}(0).
\end{equation}
\item\label{item:Extra Inside the Sphere case 2} For $d \geq 3$, $s = d-4$, $R> 0$, and $v$ is $\mathcal{C}^2$ in the extended sense on $[0,R^2]$ such that conditions \ref{item:NecCond SP} and \ref{item:NecCond BC0} of
Theorem~\ref{thm:Necessary Condition Riesz}
are satisfied.
In addition, there is some $\lambda_2 \in [0, 1)$ such that
$v''(R^2 \lambda)$ is nonpositive on $[0,\lambda_2)$ and
nonnegative on $(\lambda_2, 1]$.

\end{enumerate}

\end{lemma}

\begin{proof}
We handle each of the cases separately.

\underline{Case} \ref{item:General Inside the Sphere case 1}: Since $k = \lceil \frac{d-s}{2} \rceil$ (note that $s < d-4$ means that $k \geq 3$),   $(-1)^k f^{(k)}$ is nonpositive on $[0,1]$, by  \eqref{eq:Derivatives of f_s,d,R,v in sphere},  \eqref{eq:positivity of 2F1 on [0,1)}, and \eqref{eq:HG1pos}. This along with \eqref{eq:Completely_monotonic_external_fields_-2_s_d-4_inside_sphere_eq 1}, \eqref{eq:Completely monotonic_external_fields_-2_s_d-4_inside_sphere_eq 2}, and Theorem~\ref{thm:Necessary Condition Riesz}\ref{item:NecCond SP}
imply that $f$ is half-monotone of order $(k_0,k)$ at $1$ on $[0,1]$. By Proposition~\ref{prop:half-monotone implies unimodal}, $f$ is unimodal on $[0,1]$. Thus, its global minimum on $[0,1]$ must occur at an endpoint.
From Theorem~\ref{thm:Necessary Condition Riesz}\ref{item:NecCond BC0},
we know $f(0) \geq f(1)$, giving us our claim.

\vspace{1ex}
\underline{Case} \ref{item:General Inside the Sphere case 2}:
Since $k \leq \lfloor \frac{d-s}{2} \rfloor$,   $(-1)^k f^{(k)}$ is nonnegative on $[0,1]$, which follows from  \eqref{eq:Derivatives of f_s,d,R,v in sphere}, \eqref{eq:positivity of 2F1 on [0,1)}, and \eqref{eq:HG1pos}. Thus, $-f$ is half-monotone of order $(1,k)$ at $1$ on $[0,1]$, which follows from \eqref{eq:Pos_derivative_-2_s_d-4_inside_sphere} and
Theorem~\ref{thm:Necessary Condition Riesz}\ref{item:NecCond SP}.
By Proposition~\ref{prop:half-monotone implies unimodal},
$f$ is decreasing on $[0,1]$, giving us our claim.

\vspace{1ex}
\underline{Case} \ref{item:Extra Inside the Sphere case 1}:
When $k=3$, the claim follows from Lemma \ref{lem:General Inside the Sphere}. Now, consider the case when $k\geq 4$. Since $d-4 < s < d-3$,  we see that $f^{(\ell)}(1) := \lim\limits_{\lambda \rightarrow 1^-} f^{(\ell)}(\lambda)$ is $ - \infty$ for $3 \leq \ell \leq k$, due to \eqref{eq:Derivatives of f_s,d,R,v in sphere}, \eqref{eq:HG1neg}, and the fact that $v^{(\ell)}(R^2) < \infty$. Combining this with \eqref{eq:Derivatives of f_s,d,R,v in sphere}, \eqref{eq:positivity of 2F1 on [0,1)}, and the assumption $v^{(k)}(R^2 \lambda) \leq 0$ on $[0,1]$, we see that $f^{(k)} \leq 0$ on $[0,1]$. Let $\varphi(\lambda):=f''(1-\lambda)$. Then, $\varphi$ is half-monotone of order $(k-2,k-2)$ at $1$. By Proposition \ref{prop:half-monotone implies unimodal}, $\varphi$ and also $f''$ are unimodal on $[0,1]$. 
 Theorem~\ref{thm:Necessary Condition Riesz}\ref{item:NecCond DD}
 gives us that $f^{(2)}(1) \geq 0$, so $-f'$ is unimodal on $[0,1]$.
 This in turn implies that $f$ is unimodal on $[0,1]$ due to the fact that $f'(1)=0$ from
 Theorem \ref{thm:Necessary Condition Riesz}\ref{item:NecCond SP}.
 Theorem \ref{thm:Necessary Condition Riesz}\ref{item:NecCond BC0}
 now gives us our claim.

\vspace{1ex}
\underline{Case} \ref{item:Extra Inside the Sphere case 2}: Note, from \eqref{eq:Derivatives of f_s,d,R,v in sphere}, $f^{(2)}(\lambda) = R^4 v^{(2)}(R^2 \lambda)$ on $[0,1]$. Thus, $f'$ is decreasing on $[0, \lambda_2)$ and increasing on $(\lambda_2, 1]$, and since $f'(1) = 0$ (due to 
Theorem~\ref{thm:Necessary Condition Riesz}\ref{item:NecCond SP}),
there is some $\lambda_1 \in [0, 1)$ such that $f'$ is nonnegative on $[0, \lambda_1)$ and nonpositive on $(\lambda_1, 1]$, so $f$ is increasing and decreasing on those intervals, respectively. Thus, the minimum can only occur at $0$ or $1$, and our claim then follows from
Theorem~\ref{thm:Necessary Condition Riesz}\ref{item:NecCond BC0}.

\end{proof}

We now determine some sufficient conditions for behavior outside of the sphere. In what follows, $q$, $y_{s,d}$, and $g$ are as in \eqref{eq:Def of inverse external field}, \eqref{eq: y_s,d def}, and \eqref{eq:def of g_s,d,q}, respectively.
\begin{lemma}\label{lem:Extra Outside the Sphere}
If one of the following conditions hold, then $f$ achieves its infimum on $[1,\infty)$ at $1$.
\begin{enumerate}[label=\textnormal{(\alph*)}]
\item\label{item:General Outside the Sphere case 1}
For $-2 < s < d-4$, $R> 0$, $k = \lceil \frac{d-s}{2} \rceil$, and $q$ is  $\mathcal{C}^k$ in the extended sense on $[0,1]$ such that 
$q^{(k)}(\kappa) \geq 0$ on $[0,1]$.0
In addition,
Theorem~\ref{thm:Necessary Condition Riesz}\ref{item:NecCond SP}
is satisfied, and for  $\ell \in \{1, ..., k -1 \}$, $(-1)^{\ell} g^{(\ell)}(1) \geq 0$, that is
\begin{equation}\label{eq:Completely_monotonic_external_fields_-2_s_d-4_outside_sphere_eq_1}
   (-1)^{\ell} q^{(\ell)}(1) \geq (-1)^{\ell+1} y_{s,d}^{(\ell)}(1).
\end{equation}    
\item\label{item:General Outside the Sphere case 2}
For $-2 < s \leq d-4$, $R> 0$, and for some $2 \leq k \leq \lfloor \frac{d-s}{2} \rfloor$, $q$  is $\mathcal{C}^k$ in the extended sense on $[0, 1]$, with $(-1)^k q^{(k)}\leq0$ on $[0,1]$. In addition,
Theorem~\ref{thm:Necessary Condition Riesz}\ref{item:NecCond SP}
and \ref{item:NecCond BCinf}
are satisfied and there is some $k_0 \in \{2, ..., k\}$
such that for $k_0 \leq \ell < k$, $(-1)^{\ell} g^{(\ell)}(1) \leq 0$,
that is
\begin{equation}\label{eq:Completely monotonic external fields, -2 < s < d-4, ic external fields, -2 < s < d-4, eq 1}
   (-1)^{\ell+1} q^{(\ell)}(1) \geq (-1)^{\ell} y_{s,d}^{(\ell)}(1),
\end{equation}
and for $1 \leq \ell < k_0$, $(-1)^{\ell} g^{(\ell)}(1) \geq 0$, that is
\begin{equation}\label{eq:Completely_monotonic_external_fields_-2_s_d-4_outside sphere_eq_2}
   (-1)^{\ell+1} q^{(\ell)}(1) \leq (-1)^{\ell} y_{s,d}^{(\ell)}(1).
\end{equation}

\item\label{item:Extra Outside the Sphere case 1}
For $-2 < s < d-4$, $R > 0$, $k \in 2\mathbb{N}+1 \cup [3, \frac{d-s}{2})$, and  $v$ is $\mathcal{C}^k$ in the extended sense on $[R^2, \infty)$ such that $v^{(k)}$ is nonnegative on $[R^2, \infty)$. In addition,
Theorem \ref{thm:Necessary Condition Riesz}\ref{item:NecCond SP}
and \ref{item:NecCond BCinf} are satisfied,
and there is some $k_0 \in \{ 3, ..., k\}$ such that for  $k_0 \leq \ell < k$, $f^{(\ell)}(1) \leq 0$, that is
\begin{equation}\label{eq:Outside_sphere_odd_case_-2_s_d-4_eq_1}
        h_{s,d}^{(\ell)}(1) \leq  - R^{2\ell + s} v^{(\ell)}(R^2) 
     \end{equation}
and for $2 \leq \ell < k_0$, $f^{(\ell)}(1) \geq 0$, that is
\begin{equation}\label{eq:Outside_sphere_odd_case_-2_s_d-4_eq_2}
    h_{s,d}^{(\ell)}(1) \geq  - R^{2\ell + s} v^{(\ell)}(R^2).
\end{equation}

\item\label{item:Extra Outside the Sphere case 2} For $d \geq 3$, $s = d-4$, $R> 0$, and $v$ is $\mathcal{C}^3$ in the extended sense on $[R^2, \infty)$ such that
Theorem~\ref{thm:Necessary Condition Riesz}\ref{item:NecCond SP}
and \ref{item:NecCond BCinf} are satisfied.
In addition, $g'(1) \leq 0$ and there is some $\kappa_2 \in [0, 1)$ such that $q''(\kappa)$ is nonpositive on $[0,\kappa_2)$ and nonnegative on $(\kappa_2, 1]$.

\item\label{item:Extra Outside the Sphere case 3} For $d \geq 3$, $s = d-4$, $R> 0$, and $v$ is $\mathcal{C}^3$ in the extended sense on $[R^2, \infty)$ such that
Theorem~\ref{thm:Necessary Condition Riesz}\ref{item:NecCond SP}
and \ref{item:NecCond BCinf} are satisfied.
In addition, $g'(1) \leq 0$, $g'(0) \geq 0$, and there is some $\kappa_2 \in [0, 1)$ such that $q''(\kappa)$ is nonnegative on $[0,\kappa_2)$ and nonpositive on $(\kappa_2, 1]$.

\item\label{item:Extra Outside the Sphere case 4} For $d \geq 3$, $s = d-4$, $R> 0$, and $v$ is $\mathcal{C}^3$ in the extended sense on $[R^2, \infty)$ such that  Theorem \ref{thm:Necessary Condition Riesz}\ref{item:NecCond SP}
is satisfied. In addition, there is some $\kappa_2 \in (0, 1)$ such that $q''(\kappa)$ is nonnegative on $[0,\kappa_2)$ and nonpositive on $(\kappa_2, 1]$ and $g'(\kappa_2) \leq 0$. 

\end{enumerate}

\end{lemma}

\begin{proof}
We handle each case separately.

\underline{Case} \ref{item:General Outside the Sphere case 1}:
Since $k = \lceil \frac{d-s}{2} \rceil$ (note that since $s < d-4 $, $k \geq 3$),  $(-1)^k g^{(k)}$ is nonnegative on $[0,1]$, by \eqref{eq:Derivatives of g_s,d,q}, \eqref{eq:positivity of 2F1 on [0,1)}, and \eqref{eq:HG1pos}-\eqref{eq:HG1neg} (to be precise, we use \eqref{eq:HG1pos} for when $s<d-5$ and $s\neq d-6$; \eqref{eq:HG1zer} for when $s=d-5$ or $s=d-6$; \eqref{eq:HG1neg} for when $d-5<s<d-4$). Combining this with \eqref{eq:Completely_monotonic_external_fields_-2_s_d-4_outside_sphere_eq_1}, we have $-g$ is half-monotone of order $(1,k)$ at $1$. By Proposition \ref{prop:half-monotone implies unimodal}, $g$ is decreasing on $[0,1]$. From \eqref{eq:equivalence with f and q} and
 Theorem~\ref{thm:Necessary Condition Riesz}\ref{item:NecCond SP},
 we have $g(1)=2R^sf'(1)=0$, which implies $g\geq 0$ on $[0,1]$.
 Thus, $f$ must be increasing on $[1,\infty)$ by \eqref{eq:equivalence with f and q}, giving us our claim.

\vspace{1ex}
\underline{Case} \ref{item:General Outside the Sphere case 2}:
Since $1 \leq k \leq \lfloor \frac{d-s}{2} \rfloor$,  $(-1)^k g^{(k)}\leq 0$ on $[0,1)$, by \eqref{eq:Derivatives of g_s,d,q}, \eqref{eq:positivity of 2F1 on [0,1)}, and \eqref{eq:HG1pos}.  Combining this with \eqref{eq:Completely monotonic external fields, -2 < s < d-4, ic external fields, -2 < s < d-4, eq 1} and \eqref{eq:Completely_monotonic_external_fields_-2_s_d-4_outside sphere_eq_2}, we have $g$ is half-monotone of order $(k_0,k)$ at $1$. By Proposition \ref{prop:half-monotone implies unimodal}, $g$ is unimodal on $[0,1]$. On the other hand, we have $g(1)=2R^sf'(1)=0$, which follows from \eqref{eq:equivalence with f and q} and
Theorem~\ref{thm:Necessary Condition Riesz}\ref{item:NecCond SP}.
Thus, there exists some $\kappa_0 \in [0,1]$ such that $g\leq 0$ on $[0, \kappa_0)$ and $g\geq 0$ on $(\kappa_0, 1]$. Due to \eqref{eq:equivalence with f and q}, $f'\geq 0$ on $[1, \kappa_0^{-1})$ and $f'\leq 0$ on $( \kappa_0^{-1}, \infty)$, so $f$ is unimodal on $[1,\infty)$. Note that we interpret $\kappa_0^{-1}$ as $\infty$ if $\kappa_0 = 0$.
Our claim now follows from
Theorem~\ref{thm:Necessary Condition Riesz}\ref{item:NecCond BCinf}.

\vspace{1ex}
\underline{Case} \ref{item:Extra Outside the Sphere case 1}:
We see, due to \eqref{eq:Derivatives_of_f_s_d_R_v_outside_sphere}, \eqref{eq:positivity of 2F1 on [0,1)}, and \eqref{eq:HG1pos}, that $f^{(k)} \leq 0$ on $[1,\infty)$. This could be: For arbitrary $a>1$, let $\varphi_a(\xi):=f(a-(a-1)\xi)$ be a function defined on $[0,1]$. Then, $\varphi_a$ is half-monotone of order $(k_0,k)$ at $1$, which follows from
Theorem \ref{thm:Necessary Condition Riesz}\ref{item:NecCond SP}, \eqref{eq:Outside_sphere_odd_case_-2_s_d-4_eq_1}, and \eqref{eq:Outside_sphere_odd_case_-2_s_d-4_eq_2}.
Thus, $\varphi_a$ is unimodal on $[0,1]$ by Proposition \ref{prop:half-monotone implies unimodal}. This shows that $f$ is unimodal on $[1,a]$ for all $a>1$. In other words, $f$ is unimodal on $[1,\infty)$. Our claim now follows from
Theorem~\ref{thm:Necessary Condition Riesz}\ref{item:NecCond BCinf}.

\vspace{1ex}
\underline{Case} \ref{item:Extra Outside the Sphere case 2}:
From \eqref{eq:Derivatives of g_s,d,q}, $g^{(2)}(\kappa) = q^{(2)}(\kappa)$. Thus, $g'$ is decreasing on $[0, \kappa_2)$ and increasing on $(\kappa_2, 1]$. Since $g'(1) \leq 0$, there is some $\kappa_1 \in [0, 1)$ such that $g'$ is nonnegative on $[0, \kappa_1)$ and nonpositive on $(\kappa_1, 1]$, so $g$ is increasing and decreasing on those intervals, respectively. Using \eqref{eq:equivalence with f and q} and
Theorem~\ref{thm:Necessary Condition Riesz}\ref{item:NecCond SP},
we see that there must be some $\lambda_1 \in (1, \infty]$ such that $f'$ is nonnegative on $[1, \lambda_1)$ and nonpositive on $(\lambda_1, \infty)$, meaning that $f$ is increasing and decreasing on those same intervals, respectively.
Our claim now follows from 
Theorem~\ref{thm:Necessary Condition Riesz}\textnormal{\ref{item:NecCond BCinf}}.

\vspace{1ex}
\underline{Case} \ref{item:Extra Outside the Sphere case 3}:
From \eqref{eq:Derivatives of g_s,d,q}, $g^{(2)}(\kappa) = q^{(2)}(\kappa)$. Thus, $g'$ is increasing on $[0, \kappa_2)$ and decreasing on $(\kappa_2, 1]$. Since $g'(1) \leq 0$ and $g'(0) \geq 0$, there is some $\kappa_1 \in [0, 1)$ such that $g'$ is nonnegative on $[0, \kappa_1)$ and nonpositive on $(\kappa_1, 1]$, so $g$ is increasing and decreasing on those intervals, respectively. By \eqref{eq:equivalence with f and q} and
Theorem \ref{thm:Necessary Condition Riesz}\ref{item:NecCond SP},
we see that $g(1) = 0$, so there exists some $\kappa_0 \in [0,1)$ such that $g_{d-4, d,q}$ is nonpositive on $[0,\kappa_0)$ and nonnegative on $(\kappa_0, 1]$. Using \eqref{eq:equivalence with f and q},  we have that $f'$ is nonnegative on $[1, \kappa_0^{-1})$ and nonpositive on $(\kappa_0^{-1}, \infty)$, meaning that $f$ is increasing and decreasing on those same intervals, respectively. Our claim now follows from
Theorem~\ref{thm:Necessary Condition Riesz}\ref{item:NecCond BCinf}.

\vspace{1ex}
\underline{Case} \ref{item:Extra Outside the Sphere case 4}:
From \eqref{eq:Derivatives of g_s,d,q}, $g^{(2)}(\kappa) = q^{(2)}(\kappa)$. Thus, $g'$ is increasing on $[0, \kappa_2)$ and decreasing on $(\kappa_2, 1]$, achieving its maximum at $\kappa_2$. Thus, $g'$ is nonpositive on $[0,1]$, so $g$ is decreasing. Using \eqref{eq:equivalence with f and q} and
Theorem~\ref{thm:Necessary Condition Riesz}\ref{item:NecCond SP},
we see that $g(1) = 0$, so $g$ is nonnegative on $[0,1]$. Employing \eqref{eq:equivalence with f and q}, this means that $f'$ is nonnegative on $[1, \infty)$, so $f$ is increasing, which finishes the proof.
\end{proof}

\fi

\bibliographystyle{abbrv}

\bibliography{riesz}

\end{document}